\pdfoutput=1
\documentclass[hidelinks]{amsart}
\usepackage{mathtools}
\usepackage{amssymb}
\usepackage{xparse}
\usepackage{xr-hyper}
\usepackage{array}

\usepackage{bm,bbm}
\usepackage{mathrsfs}
\usepackage{stmaryrd}
\usepackage[scr=boondoxo]{mathalfa}

\usepackage[centering, footskip=6mm]{geometry}

\usepackage{xpatch}
\usepackage{mdframed}

\makeatletter
\xpatchcmd{\endmdframed}
  {\aftergroup\endmdf@trivlist\color@endgroup}
  {\endmdf@trivlist\color@endgroup\@doendpe}
  {}{}
\makeatother

\usepackage{iftex}

\ifpdf
\usepackage[hypertexnames=false]{hyperref}
\else
\usepackage[dvipdfmx,hypertexnames=false]{hyperref}
\fi

\usepackage[nameinlink]{cleveref}

\usepackage{etoolbox}
\makeatletter
\patchcmd\@thm
  {\let\thm@indent\indent}{\let\thm@indent\noindent}%
  {}{}
\expandafter\patchcmd\csname\string\proof\endcsname
  {\normalparindent}{0pt }{}{}
\makeatother

\numberwithin{equation}{section}

\makeatletter
\renewcommand\th@plain{\slshape}
\makeatother
\newtheoremstyle{plain}
  {1ex}
  {1ex}
  {\slshape}
  {}
  {\sffamily\bfseries}
  {.}
  {.5em}
  {}
\theoremstyle{plain}
\newtheorem{theorem}{Theorem}

\newtheorem{lemma}{Lemma}

\newtheorem*{claim*}{Claim}

\newtheoremstyle{definition}
  {1ex}
  {1ex}
  {\normalfont}
  {}
  {\sffamily\bfseries}
  {.}
  {.5em}
  {}
\theoremstyle{definition}

\newtheorem{definition}{Definition}

\newtheorem*{remark*}{Remark}

\crefname{section}{Section}{Sections}

\crefname{theorem}{Theorem}{Theorems}
\crefname{pretheorem}{Theorem}{Theorems}
\crefname{corollary}{Corollary}{Corollaries}
\crefname{lemma}{Lemma}{Lemmas}
\crefname{proposition}{Proposition}{Propositions}
\crefname{claim}{Claim}{Claims}
\crefname{definition}{Definition}{Definitions}
\crefname{notation}{Notation}{Notations}
\crefname{problem}{Problem}{Problems}
\crefname{question}{Question}{Questions}
\crefname{note}{Note}{Notes}
\crefname{remark}{Remark}{Remarks}
\crefname{example}{Example}{Examples}

\crefname{enumi}{}{}
\crefname{enumii}{}{}
\crefname{enumiii}{}{}

\crefformat{equation}{{\upshape#2(#1)#3}}

\expandafter\def\csname ver@etex.sty\endcsname{3000/12/31}

\usepackage{autonum}

\makeatletter
\newcommand{\restore@Environment}[1]{%
  \AtBeginDocument{%
    \csletcs{#1*}{#1}%
    \csletcs{end#1*}{end#1}%
  }%
}
\forcsvlist\restore@Environment{alignat,equation,gather,multline,flalign,align}
\makeatother

\usepackage{enumitem}
\setlist{leftmargin=20pt}
\setlist[enumerate]{label=\textup{(\roman*)}}

\ifpdf
\usepackage{graphicx}
\else
\usepackage[dvipdfmx]{graphicx}
\fi
\usepackage{tikz}
\usetikzlibrary{arrows.meta}
\usetikzlibrary{tikzmark}
\usetikzlibrary{cd}
\tikzcdset{ampersand replacement=\&}
\tikzcdset{
  cells={font=\everymath\expandafter{\the\everymath\displaystyle}},
}

\allowdisplaybreaks

\newcommand{\proofitem}[1]{\noindent\cref{#1}.}
\newcommand{\prooftitle}[1]{\noindent\textsf{#1}}

\newcommand{\textdf}[1]{\textsl{#1}}

\let\tmp\phi
\let\phi\varphi
\let\varphi\tmp
\let\tmp\epsilon
\let\epsilon\varepsilon
\let\varepsilon\tmp

\renewcommand{\leadsto}{\DOTSB\;\rightsquigarrow\;}

\newcommand{\midmid}{\mathrel{}\middle|\mathrel{}}
\newcommand{\semicolon}{\nobreak\mskip2mu\mathpunct{}\nonscript
\mkern-\thinmuskip{;}\mskip6muplus1mu\relax}
\renewcommand{\subset}{\subseteq}

\renewcommand{\mod}[1]{(\mathrm{mod}\ #1)}

\renewcommand{\and}{\quad\text{and}\quad}

\makeatletter
\NewDocumentCommand{\xsideset}{mmme{_^}}{%
\mathop{%
\settowidth{\dimen0}{$\m@th\displaystyle#3$}%
\dimen0=.5\dimen0
\settowidth{\dimen2}{$%
\m@th\displaystyle#3%
\IfValueT{#4}{_{#4}}%
\IfValueT{#5}{^{#5}}%
$}%
\dimen2=.5\dimen2
\advance\dimen2 -\dimen0
\sbox6{\scriptspace\z@$\displaystyle{\vphantom{#3}}#1$}
\sbox8{\scriptspace\z@$\displaystyle{\vphantom{#3}}#2$}
\ifdim\wd6>\dimen2 \kern\dimexpr\wd6-\dimen2\relax\fi
{%
\mathop{\llap{\copy6}{\displaystyle#3}\rlap{\copy8}}\limits%
\IfValueT{#4}{_{#4}}%
\IfValueT{#5}{^{#5}}%
}%
\ifdim\wd8>\dimen2 \kern\dimexpr\wd8-\dimen2\relax\fi
}%
}
\makeatother

\DeclareMathOperator{\rev}{rev}
\newcommand{\arev}{\overleftarrow}
\DeclareMathOperator{\len}{len}

\begin{document}

\title[The Zsiflaw--Legeis theorem for arbitrary bases]
{The Zsiflaw--Legeis theorem for arbitrary bases}
\author[G. Bhowmik]{Gautami Bhowmik}
\author[Y. Suzuki]{Yuta Suzuki}
\keywords{Prime numbers, arithmetic progressions, reversed radix representation.}
\subjclass{%
Primary:
11A63; 
Secondary:
11N05, 
11N69. 
}

\begin{abstract}
In this paper,
we prove analogues of the Dirichlet theorem on arithmetic progressions
and the Siegel--Walfisz theorem for the digital reverses of primes for arbitrary bases,
which the authors obtained in the previous paper but only for large bases.
The proof is based on a generalization of the result of Martin--Mauduit--Rivat (2014)
on the exponential sums over primes with the so-called ``digital'' functions.
\end{abstract}

\maketitle

\section{Introduction}
\label{sec:intro}
In this paper,
we continue our study \cite{BhowmikSuzuki:Telhcirid} of the digital reverses of prime numbers.
In our previous paper~\cite{BhowmikSuzuki:Telhcirid},
we studied the distribution of reversed primes,
the integers obtained by reading the digital representation of primes backwards, in arithmetic progressions.
In \cite{BhowmikSuzuki:Telhcirid}, the size of the base is restricted to $\ge31699$,
which was recently improved by Chourasiya and Johnston to $\ge26000$ in \cite{ChourasiyaJohnston}.
In the current paper, as is announced in several places (see, e.g.~\cite[Subsection~1.2]{ChourasiyaJohnston},
or video and slides of \cite{Suzuki:OWNS}),
we shall remove the constraint on the base to include all bases greater than or equal to $2$.

We recall the formal definition of the digital reverse.
Take $g\in\mathbb{Z}_{\ge2}$ as the base of the digital representation
and write the base-$g$ representation of $n\in\mathbb{Z}_{\ge0}$ as
\begin{equation}
\label{sec:intro:base_g_representation}
n
=
\sum_{i\ge0}\epsilon_{i}(n)g^{i}
\quad\text{with}\quad
\epsilon_{i}(n)\in\{0,\ldots,g-1\}\ \text{for all $i\in\mathbb{Z}_{\ge0}$}.
\end{equation}
We define \text{the length $\len(n)$} of $n$ by
\[
\len(n)
\coloneqq
\min\{\ell\in\mathbb{Z}_{\ge0}\mid\text{$\epsilon_{i}(n)=0$ for all $i\ge\ell$}\},
\]
i.e.\ $\len(n)$ is the number of digits of $n$
with the convention $\len(0)=0$.
For $n\in\mathbb{Z}$, we define
\begin{equation}
\label{sec:intro:def:abs_digital_reverse}
\rev(n)
\coloneqq
\sum_{0\le i<\len(n)}\epsilon_{i}(n)g^{\len(n)-i-1},
\end{equation}
which is the integer obtained by reading the base-$g$ representation of $n$ backwards.

We now have the infinitude of reversed primes
in a given arithmetic progressions for all bases:
\begin{theorem}[Telhcirid's theorem on arithmetic progressions]
\label{thm:Telhcirid}
For $g,a,q\in\mathbb{Z}$ with
\[
g\ge2,\quad
q\ge1,\quad
(a,q,g^{2}-1)=1,\quad
g\nmid(a,q),
\]
there are infinitely many primes $p$ such that $\rev(p)\equiv a\ \mod{q}$.
\end{theorem}

Note that the conditions $(a,q,g^{2}-1)=1$ and $g\nmid(a,q)$ are necessary conditions
for the infinitude stated in \cref{thm:Telhcirid}. See, e.g. Section~1 of \cite{BhowmikSuzuki:Telhcirid}.

As in \cite{BhowmikSuzuki:Telhcirid},
we have a quantitative version of \cref{thm:Telhcirid}.
In order to avoid having a wild behavior of the main term,
we just count the reversed primes of a fixed digit $L\in\mathbb{Z}$,
i.e.\ let
\[
\arev{\pi}\!_{L}(a,q)
\coloneqq
\#\{g^{L-1}\le p<g^{L}\mid\rev(p)\equiv a\ \mod{q}\}.
\]
Then, the quantitative version of \cref{thm:Telhcirid} is the following:
\begin{theorem}[Zsiflaw--Legeis theorem]
\label{thm:Zsiflaw_Legeis}
For $g,a,q\in\mathbb{Z}$ and $x\ge1$ with $g\ge2$ and $q\ge1$, we have
\[
\overleftarrow{\pi}\!_{L}(a,q)
=
\frac{\rho_{g}(a,q)}{q}
\frac{g^{L}}{\log g^{L}}
\biggl(1+O\biggl(\frac{1}{L}\biggr)\biggr)
+
O(g^{L}\exp(-c\sqrt{L}))
\]
provided
\begin{equation}
\label{cor:ZsiflawLegeis:d_cond}
q
\le
\exp(c\sqrt{L})
\end{equation}
with some constant $c>0$, where the function $\rho_{g}(a,q)$ is given by
\[
\rho_{g}(a,q)
\coloneqq
\left\{
\begin{array}{>{\displaystyle}cl}
\biggl(
1
-
\mathbbm{1}_{(q,g)\mid a}
\frac{(q,g)}{g}
\biggr)
\frac{(q,g^{2}-1)}{\phi((q,g^{2}-1))}
&\text{if $(a,q,g^{2}-1)=1$ and $g\nmid (a,q)$},\\[6mm]
0&\text{otherwise}
\end{array}
\right.
\]
and $c$ and the implicit constant depend only on $g$ and are effectively computable.
\end{theorem}

\begin{remark*}
During the final write-up of this paper,
we discovered that Dartyge, Rivat and Swaenepoel \cite[Theorem~1.4]{DRS:BV}
have independently obtained the above results.
\end{remark*}

Our previous result~\cite{BhowmikSuzuki:Telhcirid}
was proved by using the discrete circle method
in the style of Maynard~\cite{Maynard:PnP}.
In order to improve the range of the base,
we use a more direct method of Mauduit and Rivat~\cite{MauduitRivat:Gelfond}
(see also Maynard~\cite{Maynard:RestrictedDigitP}),
which was used to study the distribution of the sum of digits in arithmetic progressions over primes.
We shall reduce the problem to the estimate of the exponential sum of the form
\begin{equation}
\label{sec:intro:exponential_sum}
\sum_{n\le x}\Lambda(n)e(f(n))
\end{equation}
with a function $f$ related to some digital property of integers.
In Mauduit--Rivat~\cite{MauduitRivat:Gelfond},
$f$ is chosen to be the sum-of-digit function,
and in our case, $f$ will be the digital reverse with a certain modification.
The method of Mauduit and Rivat to bound the exponential sum \cref{sec:intro:exponential_sum}
has been generalized by Martin--Mauduit--Rivat~\cite{MartinMauduitRivat:Digital}
by replacing the sum-of-digit function
by a class of more general arithmetic functions called ``digital functions''.
An arithmetic function $f\colon\mathbb{Z}_{\ge0}\to\mathbb{R}$
defined with a map $\alpha\colon\{0,\ldots,g-1\}\to\mathbb{R}$ by
\begin{equation}
\label{sec:intro:def:digital}
f(n)
\coloneqq
\sum_{0\le i<\len(n)}\alpha(\epsilon_{i}(n))
\end{equation}
is said to be ``digital'' by Martin--Mauduit--Rivat.
However, our digital reverse $\rev$ (or its modification) is not a digital function in the above sense
since any arithmetic function $f$ satisfying \cref{sec:intro:def:digital} enjoys the bound
\[
f(n)
\le
\biggl(\max_{0\le m<g}\alpha(m)\biggr)\len(n)
\ll
\log n
\]
while $rev$ does not. We thus generalize the notion of digital functions
one step more to the following ``weakly'' digital functions
by allowing the function $\alpha$ now depending on the ``position'' $i$:
\begin{definition}
\label{def:}
Let $\mathscr{A}_{g}$ be the set of sequences of maps
\[
\bm{\alpha}=(\alpha_{i}\colon\{0,\ldots,g-1\}\to\mathbb{R})_{i=0}^{\infty}.
\]
For $\bm{\alpha}\in\mathscr{A}_{g}$ and $\lambda\in\mathbb{Z}_{\ge0}$, we define a map
$f_{\lambda}=f_{\lambda,\bm{\alpha}}\colon\mathbb{Z}_{\ge0}\to\mathbb{R}$
by
\[
f_{\lambda,\bm{\alpha}}
\coloneqq
\sum_{0\le i<\lambda}\alpha_{i}(\epsilon_{i}(n)).
\]
We call such a sequence of functions $(f_{\lambda})_{\lambda=0}^{\infty}$
a \textdf{weakly digital function generated by the seed $\bm{\alpha}$}.
Note that $f_{\lambda}$ has a period $g^{\lambda}$.
\end{definition}

The resulting estimate for the exponential sum will be stated as \cref{thm:exp_sum} below
(see also \cref{def:gamma} and \cref{def:sigma_j_alpha} for some definitions).
For the proof of \cref{thm:exp_sum},
we follow the argument of Martin--Mauduit--Rivat~\cite{MartinMauduitRivat:Digital},
but we need to take care of the shifts of $\bm{\alpha}\in\mathscr{A}_{g}$
and a modification of the auxiliary function $\Psi$ given in \cref{def:Psi} below
(cf. the formula (19) of \cite{MartinMauduitRivat:Digital}).

This paper is organized as follows.
We first prove some basic properties of weakly digital functions
in \cref{sec:basic_wdf}.
We then prepare some lemmas on the exponential sum with weakly digital functions
in \cref{sec:exp_sum}, especially the $L^{1}$, $L^{\infty}$ and hybrid bounds.
In \cref{sec:TypeI} and \cref{sec:TypeII},
we prepare estimates for the Type I and Type II sums.
In \cref{sec:exp_sum_over_prime_wdf}, we shall prove \cref{thm:exp_sum}.
Finally, we deduce \cref{thm:Telhcirid}, \cref{thm:Zsiflaw_Legeis} and \cref{thm:Zsiflaw_Legeis_pure}
in \cref{sec:digital_reverse}.
We expect to develop further results on the distribution of the weakly digital function over primes in a later paper.

\section{Notation}
\label{sec:notation}
Besides those introduced in the main body,
we use the following notations.

The letter $p$ is reserved for prime numbers.
The letter $c$ is used for positive constants
which can take different values line by line.
The arithmetic function $\phi(n)$ is the Euler totient function,
$\mu(n)$ is the M\"{o}bius function,
and $\tau(n)$ is the divisor function, i.e.\ the number of positive divisors of $n$.
For integers $n_{1},\ldots,n_{r}$,
we write $(n_{1},\ldots,n_{r})$ for the greatest common divisor of $n_{1},\ldots,n_{r}$.

For a real number $x$, we let $e(x)\coloneqq\exp(2\pi ix)$, while
$[x]$ denotes the integer part of $x$, i.e.\ the greatest integer $\le x$
and $\|x\|\coloneqq\min_{n\in\mathbb{Z}}|x-n|$ denotes the distance between $x$ and its nearest integer.

For a logical formula $P$,
we write $\mathbbm{1}_{P}$
for the indicator function of $P$.

If a theorem or a lemma is stated
with the phrase ``where the implicit constant depends on $a,b,c,\ldots$'',
then every implicit constant in the corresponding proof
may also depend on $a,b,c,\ldots$ without being specifically mentioned.

\section{Basic properties of the weakly digital functions}
\label{sec:basic_wdf}
As mentioned in the introduction, we need to consider the shifts of the seed $\bm{\alpha}\in\mathscr{A}_{g}$.
Therefore, for $j\in\mathbb{Z}_{\ge0}$ and $\bm{\alpha}\in\mathscr{A}_{g}$,
we define the $j$-shift $\bm{\alpha}^{[j]}=(\alpha_{i}^{[j]})_{i=0}^{\infty}\in\mathscr{A}_{g}$ of $\bm{\alpha}$ by
\[
\alpha_{i}^{[j]}=\alpha_{i+j}.
\]
We also write
\[
f_{\lambda}^{[j]}
=
f_{\lambda,\bm{\alpha}}^{[j]}
\coloneqq
f_{\lambda,\bm{\alpha}^{[j]}}.
\]
In this section, we first prove a certain ``$g$-additivity'' of weakly digital functions $f_{\lambda}^{[j]}$.

Recall that a function $f\colon\mathbb{Z}_{\ge0}\to\mathbb{C}$ is called \textsl{$g$-additive} if
\[
f(g^{k}m+n)
=
f(g^{k}m)
+
f(n)
\]
for all $k,m,n\in\mathbb{Z}_{\ge0}$ with $0\le n<g^{k}$ and \textsl{strongly $g$-additive} if
\[
f(g^{k}m+n)
=
f(m)
+
f(n)
\]
for all $k,m,n\in\mathbb{Z}_{\ge0}$ with $0\le n<g^{k}$ 
(see e.g.~the second paragraph of Gelfond~\cite{Gelfond}).
Our weakly digital functions are, not strongly $g$-additive nor $g$-additive in a strict sense,
but we have the following similar identity:

\begin{lemma}
\label{lem:wdf_additive}
For $g\in\mathbb{Z}_{\ge2}$, $\bm{\alpha}\in\mathscr{A}_{g}$
and $\lambda,j,k,m,n\in\mathbb{Z}_{\ge0}$ with $0\le k\le\lambda$ and $n<g^{k}$, we have
\[
f_{\lambda}^{[j]}(g^{k}m+n)
=
f_{\lambda-k}^{[j+k]}(m)+f_{k}^{[j]}(n).
\]
\end{lemma}
\begin{proof}
Since $n<g^{k}$, we have
\[
g^{k}m+n
=
\sum_{i\ge 0}\epsilon_{i}(m)g^{i+k}
+
\sum_{0\le i<k}\epsilon_{i}(n)g^{i}
=
\sum_{i\ge k}\epsilon_{i-k}(m)g^{i}
+
\sum_{0\le i<k}\epsilon_{i}(n)g^{i},
\]
and so
\[
\epsilon_{i}(g^{k}m+n)
=
\left\{
\begin{array}{>{\displaystyle}cl}
\epsilon_{i}(n)&\text{for $0\le i<k$},\\[2mm]
\epsilon_{i-k}(m)&\text{for $i\ge k$}.
\end{array}
\right.
\]
By recalling $0\le k\le\lambda$, we thus have
\begin{align}
f_{\lambda}^{[j]}(g^{k}m+n)
&=
\sum_{k\le i<\lambda}
\alpha_{i}^{[j]}(\epsilon_{i-k}(m))
+
\sum_{0\le i<k}
\alpha_{i}^{[j]}(\epsilon_{i}(n))\\
&=
\sum_{0\le i<\lambda-k}
\alpha_{i+k}^{[j]}(\epsilon_{i}(m))
+
f_{k}^{[j]}(n)\\
&=
\sum_{0\le i<\lambda-k}
\alpha_{i}^{[j+k]}(\epsilon_{i}(m))
+
f_{k}^{[j]}(n)
=
f_{\lambda-k}^{[j+k]}(m)
+
f_{k}^{[j]}(n).
\end{align}
This completes the proof.
\end{proof}

\begin{lemma}
\label{lem:sum_cleanup}
For $g\in\mathbb{Z}_{\ge2}$,
$\bm{\alpha}\in\mathscr{A}_{g}$,
$L\in\mathbb{Z}_{\ge0}$, $x\in\mathbb{R}_{\ge1}$ and $\beta\in\mathbb{R}$
with $1\le x\le g^{L}$, we have
\[
\biggl|
\sum_{0\le n<x}
e(f_{L}^{[j]}(n)-\beta n)
\biggr|
\le
(g-1)
\sum_{0\le\lambda\le\frac{\log\lceil x\rceil}{\log g}}
\biggl|
\sum_{0\le n<g^{\lambda}}
e(f_{\lambda}^{[j]}(n)-\beta n)
\biggr|.
\]
\end{lemma}
\begin{proof}
Let $N\coloneqq\lceil x\rceil\ge1$. We then have
\begin{equation}
\label{lem:sum_cleanup:x_to_N}
\sum_{0\le n<x}
e(f_{L}^{[j]}(n)-\beta n)
=
\sum_{0\le n<N}
e(f_{L}^{[j]}(n)-\beta n).
\end{equation}
Let $k\coloneqq[\frac{\log N}{\log g}]$ so that $g^{k}\le N<g^{k+1}$.
Since $x\le g^{L}$, we have $k\le L$. We then write
\begin{equation}
\label{lem:sum_cleanup:long_division}
N=g^{k}y+z
\quad\text{with}\quad
y\in[1,g)\cap\mathbb{Z}
\ \text{and}\ 
z\in[0,g^{k})\cap\mathbb{Z}.
\end{equation}
We prove the assertion by induction on $k\in\{0,\ldots,L\}$.
\medskip

\prooftitle{Initial case $k=0$.}
In this case, we have $N\le g-1$, and so the trivial bound gives
\[
\biggl|
\sum_{0\le n<x}
e(f_{L}^{[j]}(n)-\beta n)
\biggr|
=
\biggl|
\sum_{0\le n<N}
e(f_{L}^{[j]}(n)-\beta n)
\biggr|
\le
N
\le
g-1
\]
and so the assertion holds. 
\medskip

\prooftitle{Induction step.}
Assume that the assertion holds for the case $0\le k<K$ with $K\in\{1,\ldots,L\}$.
We shall prove the case $k=K$.
By \cref{lem:sum_cleanup:long_division}, we have
\begin{equation}
\label{lem:sum_cleanup:decomp}
\begin{aligned}
\sum_{0\le n<N}
e(f_{L}^{[j]}(n)-\beta n)
&=
\sum_{0\le m<y}
\sum_{g^{K}m\le n<g^{K}(m+1)}
e(f_{L}^{[j]}(n)-\beta(n))\\
&\hspace{0.2\textwidth}
+
\sum_{g^{K}y\le n<N}
e(f_{L}^{[j]}(n)-\beta(n))\\
&=
\sum_{0\le m<y}
\sum_{0\le n<g^{K}}
e(f_{L}^{[j]}(g^{K}m+n)-\beta(g^{K}m+n))\\
&\hspace{0.2\textwidth}
+
\sum_{0\le n<z}
e(f_{L}^{[j]}(g^{K}y+n)-\beta(g^{K}y+n))\\
&\eqqcolon
{\sum}_{1}+{\sum}_{2},\quad\text{say}.
\end{aligned}
\end{equation}
For the sum ${\sum}_{1}$, by using \cref{lem:wdf_additive} and $y\in[1,g)\cap\mathbb{Z}$, we have
\begin{align}
\label{lem:sum_cleanup:sum1}
\Bigl|
{\sum}_{1}
\Bigr|
&=
\biggl|
\sum_{0\le m<y}
e(f_{L-K}^{[j+K]}(m)-\beta g^{K}m)
\sum_{0\le n<g^{K}}
e(f_{K}^{[j]}(n)-\beta n)
\biggr|\\
&\le
(g-1)
\biggl|
\sum_{0\le n<g^{K}}
e(f_{K}^{[j]}(n)-\beta n)
\biggr|.
\end{align}
By using \cref{lem:wdf_additive}, we can bound ${\sum}_{2}$ as
\[
\Bigl|{\sum}_{2}\Bigr|
=
\biggl|
e(f_{L-K}^{[j+K]}(y)-\beta g^{K}y)
\sum_{0\le n<z}
e(f_{K}^{[j]}(n)-\beta n)
\biggr|
\le
\biggl|
\sum_{0\le n<z}
e(f_{K}^{[j]}(n)-\beta n)
\biggr|.
\]
When $z\ge1$, we have
$0\le[\frac{\log\lceil z\rceil}{\log g}]=[\frac{\log z}{\log g}]<K$
and so the induction hypothesis implies
\begin{equation}
\label{lem:sum_cleanup:sum2}
\Bigl|{\sum}_{2}\Bigr|
\le
(g-1)
\sum_{0\le\lambda<K}
\biggl|
\sum_{0\le n<g^{\lambda}}
e(f_{\lambda}^{[j]}(n)-\beta n)
\biggr|.
\end{equation}
This holds clearly if $z=0$.
On inserting \cref{lem:sum_cleanup:sum1} and \cref{lem:sum_cleanup:sum2} into \cref{lem:sum_cleanup:decomp}
and using \cref{lem:sum_cleanup:x_to_N}, we obtain
\[
\biggl|
\sum_{0\le n<x}
e(f_{L}^{[j]}(n)-\beta n)
\biggr|
\le
(g-1)
\sum_{0\le\lambda\le K}
\biggl|
\sum_{0\le n<g^{\lambda}}
e(f_{\lambda}^{[j]}(n)-\beta n)
\biggr|,
\]
i.e.\ the assertion for the case $k=K$.
\end{proof}

\section{Exponential sums for weakly digital functions}
\label{sec:exp_sum}
We study the distribution of weakly digital functions via the discrete Fourier analysis.
For a given $\bm{\alpha}\in\mathscr{A}_{g}$, $\lambda,j\in\mathbb{Z}_{\ge0}$ and $\beta\in\mathbb{R}$,
we thus introduce a normalized exponential sum
\begin{equation}
\label{def:F_lambda}
F_{\lambda}^{[j]}(\beta)
\coloneqq
\frac{1}{g^{\lambda}}
\sum_{0\le n<g^{\lambda}}
e(f_{\lambda}^{[j]}(n)-\beta n)
\quad\text{and}\quad
F_{\lambda}(\beta)
\coloneqq
F_{\lambda}^{[0]}(\beta).
\end{equation}
Note that the function
$\beta\mapsto F_{\lambda}^{[j]}(\beta)$
has period $1$.
For $\bm{\alpha}\in\mathscr{A}_{g}$, $i,j\in\mathbb{Z}_{\ge0}$ and $\beta\in\mathbb{R}$,
we also use
\[
\phi_{i}^{[j]}(\beta)
\coloneqq
\biggl|\sum_{0\le n<g}e(\alpha_{i}^{[j]}(n)-\beta n)\biggr|
\quad\text{and}\quad
\phi_{i}(\beta)
\coloneqq
\phi_{i}^{[0]}(\beta).
\]
In this section, we prepare several results on the exponential sum $F_{\lambda}^{[j]}(\beta)$.

\subsection{Product formula}
\label{subsec:product_formula}
The key for studying digital properties of integers is,
as in the preceding studies, the following product formula for the exponential sum $F_{\lambda}^{[j]}(\beta)$:
\begin{lemma}
\label{lem:F_product_formula}
For $g\in\mathbb{Z}_{\ge2}$,
$\bm{\alpha}\in\mathscr{A}_{g}$,
$\lambda,j\in\mathbb{Z}_{\ge0}$ and $\beta\in\mathbb{R}$, we have
\[
|F_{\lambda}^{[j]}(\beta)|
=
\frac{1}{g^{\lambda}}
\prod_{0\le i<\lambda}
\phi_{i}^{[j]}(\beta g^{i}).
\]
\end{lemma}
\begin{proof}
By considering the base-$g$ representation of $n$ in \cref{def:F_lambda}, we have
\begin{align}
|F_{\lambda}^{[j]}(\beta)|
&=
\frac{1}{g^{\lambda}}
\biggl|
\sum_{0\le n_{0},\cdots,n_{\lambda-1}<g}
e\biggl(\sum_{0\le i<\lambda}\alpha_{i}^{[j]}(n_{i})-\sum_{0\le i<\lambda}\beta n_{i}g^{i}\biggr)
\biggr|\\
&=
\frac{1}{g^{\lambda}}
\prod_{0\le i<\lambda}
\biggl|
\sum_{0\le n<g}
e(\alpha_{i}^{[j]}(n)-\beta g^{i}n)
\biggr|.
\end{align}
By recalling the definition of $\phi_{i}^{[j]}(\beta)$, we obtain the assertion.
\end{proof}

\begin{lemma}
\label{lem:F_recursion}
For $g\in\mathbb{Z}_{\ge2}$, $\bm{\alpha}\in\mathscr{A}_{g}$,
$\lambda\in\mathbb{N}$, $j\in\mathbb{Z}_{\ge0}$ and $\beta\in\mathbb{R}$, we have
\[
|F_{\lambda}^{[j]}(\beta)|
=
|F_{\lambda-1}^{[j+1]}(\beta g)|
\times
\frac{1}{g}
\phi_{j}(\beta).
\]
\end{lemma}
\begin{proof}
By \cref{lem:F_product_formula}, we have
\[
|F_{\lambda}^{[j]}(\beta)|
=
\frac{1}{g^{\lambda-1}}
\prod_{1\le i<\lambda}
\phi_{i}^{[j]}(\beta g^{i})
\times
\frac{1}{g}\phi_{0}^{[j]}(\beta).
\]
By shifting the variable $i$ by $1$
and using $\phi_{i+1}^{[j]}=\phi_{i}^{[j+1]}$ and $\phi_{0}^{[j]}=\phi_{j}^{[0]}=\phi_{j}$,
we obtain
\begin{align}
|F_{\lambda}^{[j]}(\beta)|
&=
\frac{1}{g^{\lambda-1}}
\prod_{0\le i<\lambda-1}
\phi_{i}^{[j+1]}(\beta g\cdot g^{i})
\times
\frac{1}{g}\phi_{j}(\beta)
=
|F_{\lambda-1}^{[j+1]}(\beta g)|
\times
\frac{1}{g}
\phi_{j}(\beta).
\end{align}
This completes the proof.
\end{proof}

\subsection{The \texorpdfstring{$L^{\infty}$}{L infinity}-bound}
\label{subsec:L_infty}
We next prepare a pointwise bound for $F_{\lambda}^{[j]}(\beta)$.

\begin{lemma}
\label{lem:phi_bound}
For $g\in\mathbb{Z}_{\ge2}$, $\bm{\alpha}\in\mathscr{A}_{g}$,
$i,j,m,n\in\mathbb{Z}_{\ge0}$ with $0\le m<n\le g$
and $\beta\in\mathbb{R}$, we have
\[
\phi_{i}^{[j]}(\beta)
\le
g\exp\biggl(-\frac{8}{g}\|\alpha_{i}^{[j]}(m)-\alpha_{i}^{[j]}(n)-\beta(m-n)\|^{2}\biggr).
\]
\end{lemma}
\begin{proof}
Since $m\neq n$, we have
\begin{align}
\phi_{i}^{[j]}(\beta)
&\le
|e(\alpha_{i}^{[j]}(m)-\beta m)
+
e(\alpha_{i}^{[j]}(n)-\beta n)|
+
\biggl|
\sum_{\substack{
0\le k<g\\
k\not\in\{m,n\}
}}
e(\alpha_{i}^{[j]}(k)-\beta k)
\biggr|\\
&\le
|1+e(\alpha_{i}^{[j]}(m)-\alpha_{i}^{[j]}(n)-\beta(m-n))|
+
g-2.
\end{align}
We now use the inequality
\[
|1+e(\beta)|
=
2|\cos\pi\beta|
=
2\cos\pi\|\beta\|
\le
2(1-4\|\beta\|^{2})
\]
available for all $\beta\in\mathbb{R}$. This gives
\begin{align}
\phi_{i}^{[j]}(\beta)
&\le
2(1-4\|\alpha_{i}^{[j]}(m)-\alpha_{i}^{[j]}(n)-\beta(m-n)\|^{2})
+
g-2\\
&=
g\biggl(1-\frac{8}{g}\|\alpha_{i}^{[j]}(m)-\alpha_{i}^{[j]}(n)-\beta(m-n)\|^{2}\biggr)\\
&\le
g\exp\biggl(-\frac{8}{g}\|\alpha_{i}^{[j]}(m)-\alpha_{i}^{[j]}(n)-\beta(m-n)\|^{2}\biggr)
\end{align}
since $\frac{8}{g}\|\beta\|^{2}\le\frac{2}{g}\le1$ for any $\beta\in\mathbb{R}$.
This completes the proof.
\end{proof}

For $g\in\mathbb{Z}_{\ge2}$, $\bm{\alpha}\in\mathscr{A}_{g}$ and $i,j\in\mathbb{Z}_{\ge0}$, let us introduce
\begin{equation}
\label{def:gamma}
\gamma_{i}^{[j]}(\bm{\alpha})
\coloneqq
\frac{2\log 2}{(g-1)g^{4}(\log g)^{2}}
\sum_{0\le m<n<g}
\bigl\|\bigl(g\alpha_{i}^{[j]}(m)-\alpha_{i+1}^{[j]}(m)\bigr)
-
\bigl(g\alpha_{i}^{[j]}(n)-\alpha_{i+1}^{[j]}(n)\bigr)\bigr\|^{2}.
\end{equation}
Also, write $\gamma_{i}(\bm{\alpha})\coloneqq\gamma_{i}^{[0]}(\bm{\alpha})$.
Since $\|\beta\|\le\frac{1}{2}$ for all $\beta\in\mathbb{R}$, we have
\begin{equation}
\label{gamma_bound}
\gamma_{i}^{[j]}(\bm{\alpha})
\in
\biggl[0,\frac{\log 2}{4g^{3}(\log g)^{2}}\biggr)
\subset
\biggl[0,\frac{1}{20}\biggr)
\end{equation}
since
\[
\gamma_{i}^{[j]}(\bm{\alpha})
\le
\frac{\log 2}{2(g-1)g^{4}(\log g)^{2}}
\sum_{0\le m<n<g}1
=
\frac{\log 2}{4g^{3}(\log g)^{2}}
\le
\frac{1}{32\log 2}
<
\frac{1}{20}.
\]
We defined $\gamma_{i}^{[j]}(\bm{\alpha})$ smaller than enough for the proof of the next lemma.
However, this is only for the simplification of the subsequent arguments.
\begin{lemma}
\label{lem:consecutive_phi}
For $g\in\mathbb{Z}_{\ge2}$, $i,j\in\mathbb{Z}_{\ge0}$ and $\beta\in\mathbb{R}$, we have
\begin{align}
\bigl(
\phi_{i}^{[j]}(\beta g^{i})
\phi_{i+1}^{[j]}(\alpha g^{i+1})
\bigr)^{\frac{1}{2}}
\le
g^{1-\gamma_{i}^{[j]}(\bm{\alpha})}.
\end{align}
\end{lemma}
\begin{proof}
For any $m\neq n$, we have
\begin{align}
&g\bigl(\alpha_{i}^{[j]}(m)-\alpha_{i}^{[j]}(n)-\beta g^{i}(m-n)\bigr)
-
\bigl(\alpha_{i+1}^{[j]}(m)-\alpha_{i+1}^{[j]}(n)-\beta g^{i+1}(m-n)\bigr)\\
&=
\bigl(g\alpha_{i}^{[j]}(m)-\alpha_{i+1}^{[j]}(m)\bigr)
-
\bigl(g\alpha_{i}^{[j]}(n)-\alpha_{i+1}^{[j]}(n)\bigr).
\end{align}
Since the function $\mathbb{R}\to\mathbb{R}\semicolon x\mapsto\|x\|$ satisfies the triangle inequality, we thus get
\begin{align}
&\frac{1}{g+1}
\bigl\|\bigl(g\alpha_{i}^{[j]}(m)-\alpha_{i+1}^{[j]}(m)\bigr)
-
\bigl(g\alpha_{i}^{[j]}(n)-\alpha_{i+1}^{[j]}(n)\bigr)\bigr\|\\
&\le
\max\bigl(\|\alpha_{i}^{[j]}(m)-\alpha_{i}^{[j]}(n)-\beta g^{i}(m-n)\|,
\|\alpha_{i+1}^{[j]}(m)-\alpha_{i+1}^{[j]}(n)-\beta g^{i+1}(m-n)\|\bigr).
\end{align}
Thus, by \cref{lem:phi_bound}, we have
\[
\phi_{i}^{[j]}(\beta g^{i})
\phi_{i+1}^{[j]}(\alpha g^{i+1})
\le
g^{2}
\exp\biggl(
-
\frac{8}{g(g+1)^{2}}
\bigl\|\bigl(g\alpha_{i}^{[j]}(m)-\alpha_{i+1}^{[j]}(m)\bigr)
-
\bigl(g\alpha_{i}^{[j]}(n)-\alpha_{i+1}^{[j]}(n)\bigr)\bigr\|^{2}
\biggr).
\]
By using the inequality
\begin{align}
&\frac{8}{g(g+1)^{2}}
\max_{0\le m<n<g}
\bigl\|\bigl(g\alpha_{i}^{[j]}(m)-\alpha_{i+1}^{[j]}(m)\bigr)
-
\bigl(g\alpha_{i}^{[j]}(n)-\alpha_{i+1}^{[j]}(n)\bigr)\bigr\|^{2}\\
&\ge
\frac{16}{(g-1)g^{2}(g+1)^{2}}
\sum_{0\le m<n<g}
\bigl\|\bigl(g\alpha_{i}^{[j]}(m)-\alpha_{i+1}^{[j]}(m)\bigr)
-
\bigl(g\alpha_{i}^{[j]}(n)-\alpha_{i+1}^{[j]}(n)\bigr)\bigr\|^{2}
\end{align}
and
\[
\frac{16}{(g-1)g^{2}(g+1)^{2}}
\ge
\frac{2\log 2}{(g-1)g^{4}(\log g)^{2}},
\]
we obtain the result.
\end{proof}

For $g\in\mathbb{Z}_{\ge2}$, $\bm{\alpha}\in\mathscr{A}_{g}$, $\lambda,j\in\mathbb{Z}_{\ge0}$, we now introduce
\begin{equation}
\label{def:sigma_j_alpha}
\sigma_{\lambda}^{[j]}(\bm{\alpha})
\coloneqq
\sum_{0\le i<\lambda}
\gamma_{i}^{[j]}(\bm{\alpha})
\quad\text{and}\quad
\sigma_{\lambda}(\bm{\alpha})
\coloneqq
\sigma_{\lambda}^{[0]}(\bm{\alpha}).
\end{equation}
Note that \cref{gamma_bound} gives
\begin{equation}
\label{sigma_bound}
\sigma_{\lambda}^{[j]}(\bm{\alpha})
\in
\biggl[0,\frac{\log 2}{4g^{3}(\log g)^{2}}\lambda\biggr]
\subset
\biggl[0,\frac{1}{4g^{3}\log g}\lambda\biggr]
\subset
\biggl[0,\frac{\lambda}{20}\biggr]
\end{equation}
and we trivially have $\sigma_{0}^{[j]}(\bm{\alpha})=0$.
Also, $\sigma_{\lambda}^{[j]}(\bm{\alpha})$ is increasing with respect to $\lambda\in\mathbb{Z}_{\ge0}$.
\begin{lemma}[$L^{\infty}$-bound]
\label{lem:L_infty_general}
For $g\in\mathbb{Z}_{\ge2}$, $\bm{\alpha}\in\mathscr{A}_{g}$,
$\lambda,j\in\mathbb{Z}_{\ge0}$ and $\beta\in\mathbb{R}$, we have
\[
|F_{\lambda}^{[j]}(\beta)|
\ll
g^{-\sigma_{\lambda}^{[j]}(\bm{\alpha})},
\]
where the implicit constant depends only on $g$.
\end{lemma}
\begin{proof}
We may assume $\lambda\ge1$.
By \cref{lem:F_product_formula} with a trivial bound $\phi_{i}^{[j]}(\beta)\le g$, we have
\begin{align}
|F_{\lambda}^{[j]}(\beta)|
&\le
\frac{1}{g^{\lambda-1}}
\prod_{0\le i<\lambda-1}
\phi_{i}^{[j]}(\beta g^{i})^{\frac{1}{2}}
\prod_{1\le i<\lambda}
\phi_{i}^{[j]}(\beta g^{i})^{\frac{1}{2}}
=
\frac{1}{g^{\lambda-1}}
\prod_{0\le i<\lambda-1}
\bigl(
\phi_{i}^{[j]}(\beta g^{i})
\phi_{i+1}^{[j]}(\beta g^{i+1})
\bigr)^{\frac{1}{2}}.
\end{align}
Then, by \cref{lem:consecutive_phi} and \cref{gamma_bound}, we have
\begin{align}
\sum_{0\le i<\lambda-1}
\gamma_{i}^{[j]}(\bm{\alpha})
=
\sum_{0\le i<\lambda}
\gamma_{i}^{[j]}(\bm{\alpha})
-
\gamma_{\lambda-1}^{[j]}(\bm{\alpha})
>
\sigma_{\lambda}^{[j]}(\bm{\alpha})
-
\frac{1}{20}
\end{align}
and so we obtain the assertion.
\end{proof}

\subsection{The \texorpdfstring{$L^{1}$}{L1}-bound}
\label{subsec:L1_bound}
We next bound the discrete $L^{1}$-moment with congruence condition
\[
\sum_{\substack{
0\le h<g^{\lambda}\\
h\equiv a\ \mod{kg^{\delta}}
}}
\biggl|F_{\lambda}\biggl(\frac{h}{g^{\lambda}}\biggr)\biggr|,
\]
where $kg^{\delta}\mid g^{\lambda}$ and $g\nmid k$.
To this end, for $i\in\mathbb{Z}_{\ge0}$, $R\mid g$, $S\mid g$ and $t\in\mathbb{R}$, we introduce
\begin{equation}
\label{def:Psi}
\Psi_{i}(t,R,S)
\coloneqq
\frac{1}{g^{2}}
\sum_{0\le r<R}
\phi_{i+1}\biggl(\frac{g(t+r)}{RS}\biggr)
\sum_{0\le s<S}
\phi_{i}\biggl(\frac{t+r}{RS}+\frac{s}{S}\biggr)
\end{equation}
and prepare an estimate for $\Psi_{i}(t,R,S)$.

We first prove some basic lemmas on the moments of $\phi_{i}$.
\begin{lemma}
\label{lem:L2_orthogonality}
For $g\in\mathbb{Z}_{\ge2}$, $\bm{\alpha}\in\mathscr{A}_{g}$,
$i\in\mathbb{Z}_{\ge0}$, $R\in\mathbb{N}$, $a\in\mathbb{Z}$ and $\beta\in\mathbb{R}$
with $R\mid g$ and $(R,a)=1$,
\[
\max_{\beta\in\mathbb{R}}
\sum_{0\le r<R}
\phi_{i}\biggl(\beta+\frac{ar}{R}\biggr)^{2}
\le
g^{2}.
\]
\end{lemma}
\begin{proof}
By expanding the square and using the orthogonality with $(R,a)=1$, we have
\begin{align}
\sum_{0\le r<R}
\phi_{i}\biggl(\beta+\frac{ar}{R}\biggr)^{2}
&=
\sum_{0\le m,n<g}
e(\alpha_{i}(m)-\alpha_{i}(n)-\beta(m-n))
\sum_{0\le r<R}
e\biggl(-\frac{ar}{R}(m-n)\biggr)\\
&\le
R
\sum_{\substack{
0\le m,n<g\\
m\equiv n\ \mod{R}
}}
1
=
g^{2}.
\end{align}
This completes the proof.
\end{proof}

\begin{lemma}
\label{lem:L4_moment}
For $g\in\mathbb{Z}_{\ge2}$, $\bm{\alpha}\in\mathscr{A}_{g}$,
$i\in\mathbb{Z}_{\ge0}$, $U\in\mathbb{N}$ and $\beta\in\mathbb{R}$ with $U\ge 2g-1$, we have
\[
\sum_{0\le u<U}
\phi_{i}\biggl(\beta+\frac{u}{U}\biggr)^{4}
=
U\sum_{|h|<g}
\biggl|
\sum_{0\le n,n+h<g}e(\alpha_{i}(n+h)-\alpha_{i}(n))
\biggr|^{2}.
\]
\end{lemma}
\begin{proof}
By expanding the square, we have
\[
\phi_{i}\biggl(\beta+\frac{u}{U}\biggr)^{2}
=
\sum_{0\le m,n<g}
e\biggl(\alpha_{i}(m)-\alpha_{i}(n)-\biggl(\beta+\frac{u}{U}\biggr)(m-n)\biggr).
\]
By writing $h=m-n$, we can rewrite the right-hand side as
\[
\phi_{i}\biggl(\beta+\frac{u}{U}\biggr)^{2}
=
\sum_{|h|<g}
c(h)
e\biggl(-\biggl(\beta+\frac{u}{U}\biggr)h\biggr)
\quad\text{with}\quad
c(h)
\coloneqq
\sum_{\substack{
0\le m,n<g\\
h=m-n
}}
e(\alpha_{i}(m)-\alpha_{i}(n)).
\]
By the orthogonality, we thus have
\begin{equation}
\label{lem:L4_moment:prefinal}
\begin{aligned}
\sum_{0\le u<U}
\phi_{i}\biggl(\beta+\frac{u}{U}\biggr)^{4}
&=
\sum_{|h_{1}|,|h_{2}|<g}
c(h_{1})\overline{c(h_{2})}
e(-\beta(h_{1}-h_{2}))
\sum_{0\le u<U}
e\biggl(-\frac{u}{U}(h_{1}-h_{2})\biggr)\\
&=
U\sum_{\substack{
|h_{1}|,|h_{2}|<g\\
h_{1}\equiv h_{2}\ \mod{U}
}}
c(h_{1})\overline{c(h_{2})}
e(-\beta(h_{1}-h_{2})).
\end{aligned}
\end{equation}
By the assumption $U\ge 2g-1$, we have
\[
|h_{1}|,|h_{2}|<g\ \text{and}\ 
h_{1}\equiv h_{2}\ \mod{U}
\implies
h_{1}=h_{2}.
\]
By using this fact in \cref{lem:L4_moment:prefinal} and noting that
\[
c(h)
=
\sum_{0\le n,n+h<g}
e(\alpha_{i}(n+h)-\alpha_{i}(n)),
\]
we obtain the assertion.
\end{proof}

We bound $\Psi_{i}(t,R,S)$ by considering two cases separately according to if $S\neq g$ or not.
\begin{lemma}
\label{lem:Psi_bound_S_neq_g}
For $i\in\mathbb{Z}_{\ge0}$, $t\in\mathbb{R}$, $g,R,S\in\mathbb{N}$
with $g\ge2$, $R\mid g$, $S\mid g$ and $(R,g/S)=1$, we have
\[
\Psi_{i}(t,R,S)^{2}
\le
\frac{1}{g/S}\biggl\lceil\frac{g}{RS}\biggr\rceil
\cdot
RS.
\]
In particular, if $R\ge2$ and $S\neq g$, we have
\[
\Psi_{i}(t,R,S)^{2}
\le
\frac{2}{3}RS.
\]
\end{lemma}
\begin{proof}
By the Cauchy--Schwarz inequality, we have
\begin{equation}
\label{lem:Psi_bound_S_neq_g:first_CS}
\Psi_{i}(t,R,S)^{2}
\le
\frac{1}{g^{4}}
\sum_{0\le r<R}
\phi_{i+1}\biggl(\frac{(g/S)t+(g/S)r}{R}\biggr)^{2}
\sum_{0\le r<R}
\biggl(
\sum_{0\le s<S}
\phi_{i}\biggl(\frac{t+r}{RS}+\frac{s}{S}\biggr)
\biggr)^{2}.
\end{equation}
By \cref{lem:L2_orthogonality} with the assumption $(R,g/S)=1$, we have
\[
\Psi_{i}(t,R,S)^{2}
\le
\frac{1}{g^{2}}
\sum_{0\le r<R}
\biggl(
\sum_{0\le s<S}
\phi_{i}\biggl(\frac{t+r}{RS}+\frac{s}{S}\biggr)
\biggr)^{2}.
\]
By applying the Cauchy--Schwarz inequality once more, we get
\begin{equation}
\label{lem:Psi_bound_S_neq_g:second_CS}
\Psi_{i}(t,R,S)^{2}
\le
\frac{S}{g^{2}}
\sum_{0\le r<R}
\sum_{0\le s<S}
\phi_{i}\biggl(\frac{t+r}{RS}+\frac{s}{S}\biggr)^{2}
\end{equation}
By expanding the square and using the orthogonality, we have
\begin{align}
\sum_{0\le s<S}
\phi_{i}\biggl(\frac{t+r}{RS}+\frac{s}{S}\biggr)^{2}
&=
\sum_{0\le m,n<g}
e\biggl(\alpha_{i}(m)-\alpha_{i}(n)-\frac{t+r}{RS}(m-n)\biggr)
\sum_{0\le s<S}
e\biggl(-\frac{s}{S}(m-n)\biggr)\\
&=
S
\sum_{\substack{
0\le m,n<g\\
m\equiv n\ \mod{S}
}}
e\biggl(\alpha_{i}(m)-\alpha_{i}(n)-\frac{t+r}{RS}(m-n)\biggr).
\end{align}
By \cref{lem:Psi_bound_S_neq_g:second_CS}, we thus have
\begin{align}
\Psi_{i}(t,R,S)^{2}
\le
\frac{S^{2}}{g^{2}}
\sum_{0\le r<R}
\sum_{\substack{
0\le m,n<g\\
m\equiv n\ \mod{S}
}}
e\biggl(\alpha_{i}(m)-\alpha_{i}(n)-\frac{t+r}{RS}(m-n)\biggr).
\end{align}
By swapping the summation, we obtain
\begin{align}
\Psi(t,R,S)^{2}
\le
\frac{S^{2}}{g^{2}}
\sum_{\substack{
0\le m,n<g\\
m\equiv n\ \mod{S}
}}
\biggl|
\sum_{0\le r<R}
e\biggl(-\frac{r}{RS}(m-n)\biggr)
\biggr|.
\end{align}
By writing $m=a+\mu S$ and $n=a+\nu S$, we get
\begin{align}
\Psi_{i}(t,R,S)^{2}
&\le
\frac{S^{2}}{g^{2}}
\sum_{0\le a<S}
\sum_{0\le\mu,\nu<\frac{g}{S}}
\biggl|
\sum_{0\le r<R}
e\biggl(-\frac{r}{RS}((a+\mu S)-(a+\nu S))\biggr)
\biggr|\\
&=
\frac{S^{3}}{g^{2}}
\sum_{0\le\mu,\nu<\frac{g}{S}}
\biggl|
\sum_{0\le r<R}
e\biggl(-\frac{r}{R}(\mu-\nu)\biggr)
\biggr|
=
\frac{RS^{3}}{g^{2}}
\sum_{\substack{
0\le \mu,\nu<\frac{g}{S}\\
\mu\equiv\nu\ \mod{R}
}}
1.
\end{align}
In the last sum,
we can write $\ell R=\mu-\nu$
by the condition $\mu\equiv\nu\ \mod{R}$
to obtain
\begin{align}
\sum_{\substack{
0\le \mu,\nu<\frac{g}{S}\\
\mu\equiv\nu\ \mod{R}
}}
1
=
\sum_{|\ell|<\frac{g}{RS}}
\sum_{\substack{
0\le \mu,\nu<\frac{g}{S}\\
\mu-\nu=\ell R
}}
1
=
\frac{g}{S}
\sum_{|\ell|<\frac{g}{RS}}
\biggl(1-\frac{|\ell|}{(\frac{g}{RS})}\biggr)
\end{align}
We thus have
\begin{equation}
\label{lem:Psi_bound_S_neq_g:prefinal}
\Psi_{i}(t,R,S)^{2}
\le
\frac{RS^{2}}{g}
\sum_{|\ell|<\frac{g}{RS}}
\biggl(1-\frac{|\ell|}{(\frac{g}{RS})}\biggr).
\end{equation}
By noting that $\lceil\frac{g}{RS}\rceil\ge\frac{g}{RS}>0$ and so $\lceil\frac{g}{RS}\rceil\ge1$,
we can estimate the last sum as
\begin{align}
\sum_{|\ell|<\frac{g}{RS}}
\biggl(1-\frac{|\ell|}{(\frac{g}{RS})}\biggr)
&=
1+2\sum_{1\le\ell\le\lceil\frac{g}{RS}\rceil-1}
\biggl(1-\frac{\ell}{(\frac{g}{RS})}\biggr)\\
&=
1+2\biggl(\biggl\lceil\frac{g}{RS}\biggr\rceil-1\biggr)
-
\frac{1}{\frac{g}{RS}}
\biggl\lceil\frac{g}{RS}\biggr\rceil
\biggl(\biggl\lceil\frac{g}{RS}\biggr\rceil-1\biggr)\\
&\le
1+\biggl(\biggl\lceil\frac{g}{RS}\biggr\rceil-1\biggr)
=
\biggl\lceil\frac{g}{RS}\biggr\rceil.
\end{align}
On inserting this formula into \cref{lem:Psi_bound_S_neq_g:prefinal},
we obtain the first assertion.

For the latter assertion,
by the assumption $S\mid g$,
we have
\[
\frac{1}{g/S}\biggl\lceil\frac{g}{RS}\biggr\rceil
\le
\frac{1}{g/S}\biggl(\frac{g}{RS}+1-\frac{1}{R}\biggr).
\]
By assuming $R\ge 2$ and $S\neq g$, we have $R,g/S\ge2$ since $S\mid g$.
Also, since $(R,g/S)=1$, we do not have $R=g/S=2$.
If $R\ge3$ and $g/S\ge2$, then we have
\[
\frac{1}{g/S}\biggl\lceil\frac{g}{RS}\biggr\rceil
\le
\frac{1}{R}
+
\frac{1}{g/S}\biggl(1-\frac{1}{R}\biggr)
\le
\frac{1}{R}
+
\frac{1}{2}\biggl(1-\frac{1}{R}\biggr)
=
\frac{1}{2R}+\frac{1}{2}
\le
\frac{2}{3}.
\]
Otherwise we have $R=2$ and $g/S\ge3$ and so
\[
\frac{1}{g/S}\biggl\lceil\frac{g}{RS}\biggr\rceil
\le
\frac{1}{R}
+
\frac{1}{g/S}\biggl(1-\frac{1}{R}\biggr)
\le
\frac{1}{R}
+
\frac{1}{3}\biggl(1-\frac{1}{R}\biggr)
=
\frac{2}{3R}+\frac{1}{3}
=
\frac{2}{3}.
\]
On inserting these estimates into the first assertion, we obtain the lemma.
\end{proof}

\begin{lemma}
\label{lem:Psi_bound_S_eq_g}
For $i\in\mathbb{Z}_{\ge0}$, $t\in\mathbb{R}$, $g,R\in\mathbb{N}$ with $R\mid g$ and $R\ge2$, we have
\[
\Psi_{i}(t,R,g)^{2}
\le
Rg(1-\Theta_{i}(\bm{\alpha})),
\]
where
\[
\Theta_{i}(\bm{\alpha})
\coloneqq
\biggl(1-\frac{1}{g}\biggr)
\biggl\{
1-\biggl(1
-
\frac{2}{g^{2}(g-1)}\sum_{1\le h<g}
\biggl|
\sum_{0\le n,n+h<g}
e(\alpha_{i}(n+h)-\alpha_{i}(n))
\biggr|^{2}\biggr)^{\frac{1}{2}}
\biggr\}.
\]
Also, for $\Theta_{i}(\bm{\alpha})$, we have
\begin{equation}
\label{lem:Psi_bound_S_eq_g:def:Theta}
\Theta_{g}
\coloneqq
\min_{\substack{
\bm{\alpha}\in\mathscr{A}_{g}\\
i\in\mathbb{Z}_{\ge0}
}}
\Theta_{i}(\bm{\alpha})
\ge
\biggl(1-\frac{1}{g}\biggr)
\biggl\{
1-\biggl(1
-
\frac{2}{g^{2}(g-1)}\biggr)^{\frac{1}{2}}
\biggr\}
>
\frac{1}{g^{3}}
>0.
\end{equation}
\end{lemma}
\begin{proof}
By putting $S=g$ in the definition \cref{def:Psi}, we have
\[
\Psi_{i}(t,R,g)
=
\frac{1}{g^{2}}
\sum_{0\le r<R}
\phi_{i+1}\biggl(\frac{t+r}{R}\biggr)
\sum_{0\le s<g}
\phi_{i}\biggl(\frac{t+r}{gR}+\frac{s}{g}\biggr).
\]
By using the Cauchy--Schwarz inequality
and \cref{lem:L2_orthogonality},
we obtain
\begin{equation}
\label{lem:Psi_bound_S_eq_g:first_CS}
\Psi_{i}(t,R,g)^{2}
\le
\frac{1}{g^{2}}
\sum_{0\le r<R}
\biggl(\sum_{0\le s<g}
\phi_{i}\biggl(\frac{t+r}{gR}+\frac{s}{g}\biggr)\biggr)^{2}.
\end{equation}
Since the function
\[
\mathbb{Z}
\to
\mathbb{R}
\semicolon
s
\mapsto
\phi_{i}\biggl(\frac{t+r}{gR}+\frac{s}{g}\biggr)
\]
has a period $g$, the sum over $s$ can be thought as a sum over residues $s\ \mod{g}$.
Thus, by \cref{lem:Psi_bound_S_eq_g:first_CS},
\begin{equation}
\label{lem:Psi_bound_S_eq_g:expand_square}
\begin{aligned}
\Psi_{i}(t,R,g)^{2}
&\le
\frac{1}{g^{2}}
\sum_{0\le r<R}
\sum_{0\le k<g}
\sum_{0\le\ell<g}
\phi_{i}\biggl(\frac{t+r}{gR}+\frac{k}{g}\biggr)
\phi_{i}\biggl(\frac{t+r}{gR}+\frac{\ell}{g}\biggr)\\
&=
\frac{1}{g^{2}}
\sum_{0\le r<R}
\sum_{0\le k<g}
\sum_{0\le\ell<g}
\phi_{i}\biggl(\frac{t+r}{gR}+\frac{k}{g}\biggr)
\phi_{i}\biggl(\frac{t+r}{gR}+\frac{k+\ell}{g}\biggr)\\
&=
\frac{1}{g^{2}}
\sum_{0\le r<R}
\sum_{0\le k<g}
\phi_{i}\biggl(\frac{t+r}{gR}+\frac{k}{g}\biggr)^{2}\\
&\hspace{0.1\textwidth}+
\frac{1}{g^{2}}
\sum_{0\le r<R}
\sum_{0\le k<g}
\sum_{1\le\ell<g}
\phi_{i}\biggl(\frac{t+r}{gR}+\frac{k}{g}\biggr)
\phi_{i}\biggl(\frac{t+r}{gR}+\frac{k+\ell}{g}\biggr).
\end{aligned}
\end{equation}
We can estimate the first sum by \cref{lem:L2_orthogonality} to get
\[
\Psi_{i}(t,R,g)^{2}
\le
R+
\frac{1}{g^{2}}
\sum_{0\le r<R}
\sum_{0\le k<g}
\sum_{1\le\ell<g}
\phi_{i}\biggl(\frac{t+r}{gR}+\frac{k}{g}\biggr)
\phi_{i}\biggl(\frac{t+r}{gR}+\frac{k+\ell}{g}\biggr).
\]
By applying the Cauchy--Schwarz inequality, this gives
\begin{equation}
\label{lem:Psi_bound_S_eq_g:second_CS}
\Psi_{i}(t,R,g)^{2}
\le
R+
\frac{1}{g^{2}}
\bigl(Rg(g-1)\Upsilon_{i}(t,R,g)\bigr)^{\frac{1}{2}},
\end{equation}
where
\[
\Upsilon_{i}(t,R,g)
\coloneqq
\sum_{0\le r<R}
\sum_{0\le k<g}
\sum_{1\le\ell<g}
\phi_{i}\biggl(\frac{t+r}{gR}+\frac{k}{g}\biggr)^{2}
\phi_{i}\biggl(\frac{t+r}{gR}+\frac{k+\ell}{g}\biggr)^{2}.
\]
We can rewrite $\Upsilon_{i}(t,R,g)$ as
\begin{equation}
\label{lem:Psi_bound_S_eq_g:Upsilon_decomp}
\begin{aligned}
\Upsilon_{i}(t,R,g)
&=
\sum_{0\le r<R}
\sum_{0\le k<g}
\sum_{0\le\ell<g}
\phi_{i}\biggl(\frac{t+r}{gR}+\frac{k}{g}\biggr)^{2}
\phi_{i}\biggl(\frac{t+r}{gR}+\frac{k+\ell}{g}\biggr)^{2}\\
&\hspace{0.4\textwidth}-
\sum_{0\le r<R}
\sum_{0\le k<g}
\phi_{i}\biggl(\frac{t+r}{gR}+\frac{k}{g}\biggr)^{4}.
\end{aligned}
\end{equation}
For the first sum in \cref{lem:Psi_bound_S_eq_g:Upsilon_decomp},
we can use \cref{lem:L2_orthogonality} to get
\begin{equation}
\label{lem:Psi_bound_S_eq_g:Upsilon_to_fourth_moment}
\Upsilon_{i}(t,R,g)
\le
Rg^{4}
-
\sum_{0\le r<R}
\sum_{0\le k<g}
\phi_{i}\biggl(\frac{t+r}{gR}+\frac{k}{g}\biggr)^{4}.
\end{equation}
The last sum can be rewritten as
\[
\sum_{0\le r<R}
\sum_{0\le k<g}
\phi_{i}\biggl(\frac{t+r}{gR}+\frac{k}{g}\biggr)^{4}
=
\sum_{0\le r<R}
\sum_{0\le k<g}
\phi_{i}\biggl(\frac{t+r+Rk}{gR}\biggr)^{4}
=
\sum_{0\le\ell<gR}
\phi_{i}\biggl(\frac{t+\ell}{gR}\biggr)^{4}.
\]
By the assumption $R\ge2$, we have $gR\ge2g-1$ and so \cref{lem:L4_moment} implies
\begin{align}
\sum_{0\le r<R}
\sum_{0\le k<g}
\phi_{i}\biggl(\frac{t+r}{gR}+\frac{k}{g}\biggr)^{4}
&=
Rg
\sum_{|h|<g}
\biggl|
\sum_{0\le n,n+h<g}
e(\alpha_{i}(n+h)-\alpha_{i}(n))
\biggr|^{2}\\
&=
Rg^{3}
+
2Rg\sum_{1\le h<g}
\biggl|
\sum_{0\le n,n+h<g}
e(\alpha_{i}(n+h)-\alpha_{i}(n))
\biggr|^{2}
\end{align}
On inserting this formula into \cref{lem:Psi_bound_S_eq_g:Upsilon_to_fourth_moment}, we get
\[
\Upsilon_{i}(t,R,g)
\le
Rg^{3}(g-1)
-2Rg\sum_{1\le h<g}
\biggl|
\sum_{0\le n,n+h<g}
e(\alpha_{i}(n+h)-\alpha_{i}(n))
\biggr|^{2}.
\]
On inserting this estimate into \cref{lem:Psi_bound_S_eq_g:second_CS}, we have
\[
\Psi_{i}(t,R,g)^{2}
\le
R
+
R(g-1)
\biggl(
1
-
\frac{2}{g^{2}(g-1)}\sum_{1\le h<g}
\biggl|
\sum_{0\le n,n+h<g}
e(\alpha_{i}(n+h)-\alpha_{i}(n))
\biggr|^{2}
\biggr)^{\frac{1}{2}}.
\]
By rewriting the right-hand side, we obtain the first result.

For the estimate of $\Theta_{i}(\bm{\alpha})$,
it suffices to pick up the contribution of the term with $h=g-1$,
which implies $n=0$.
\end{proof}

Let us introduce
\begin{equation}
\label{def:eta}
\eta_{g}
\coloneqq
\max\biggl(
\frac{1}{2}-\frac{\log\frac{3}{2}}{4\log g-2\log 2},
\frac{1}{2}+\frac{\log(1-\Theta_{g})}{4\log g}
\biggr).
\end{equation}
Since
\[
\frac{1}{2}+\frac{\log(1-\Theta_{g})}{4\log g}
\le
\frac{1}{2}-\frac{\Theta_{g}}{4\log g}
<
\frac{1}{2}-\frac{1}{4g^{3}\log g}
\]
and
\[
0.2075187
<
\frac{1}{2}-\frac{\log\frac{3}{2}}{2\log 2}
\le
\frac{1}{2}-\frac{\log\frac{3}{2}}{4\log g-2\log 2}
<
\frac{1}{2}-\frac{1}{16\log g}
<
\frac{1}{2}-\frac{1}{4g^{3}\log g},
\]
we have
\begin{equation}
\label{eta_bound}
0.2075187
<
\eta_{g}
<
\frac{1}{2}-\frac{1}{4g^{3}\log g}.
\end{equation}

\begin{lemma}
\label{lem:Psi_bound}
For $i\in\mathbb{Z}_{\ge0}$, $t\in\mathbb{R}$, $g,R,S\in\mathbb{N}$ with $R\mid g$, $S\mid g$,
$(R,g/S)=1$ and $R\ge2$, we have
\[
\Psi_{i}(t,R,S)
\le
(RS)^{\eta_{g}}.
\]
\end{lemma}
\begin{proof}
When $S\neq g$, since $R\ge2$, \cref{lem:Psi_bound_S_neq_g} gives
\[
\Psi_{i}(t,R,S)
\le
(RS)^{\frac{1}{2}-\frac{\log\frac{3}{2}}{2\log RS}}.
\]
Since $S\mid g$ and $S\neq g$ implies $S\le g/2$, we have
\[
\log RS\le\log (g^{2}/2)=2\log g-\log 2
\]
and so
\[
\Psi_{i}(t,R,S)
\le
(RS)^{\frac{1}{2}-\frac{\log\frac{3}{2}}{4\log g-2\log 2}}
\le
(RS)^{\eta_{g}}.
\]
Thus, the assertion holds if $S\neq g$.

When $S=g$, then since $R\ge2$, \cref{lem:Psi_bound_S_eq_g} gives
\[
\Psi_{i}(t,R,S)
\le
(Rg)^{\frac{1}{2}+\frac{\log(1-\Theta_{g})}{2\log Rg}}
\le
(Rg)^{\frac{1}{2}+\frac{\log(1-\Theta_{g})}{4\log g}}
\le
(RS)^{\eta_{g}}.
\]
Thus, the assertion holds even if $S=g$.
\end{proof}

\begin{lemma}
\label{lem:F_L1_moment}
For $\lambda,\delta,j\in\mathbb{Z}_{\ge0}$, $g,k\in\mathbb{N}$, $a\in\mathbb{Z}$ and $\beta\in\mathbb{R}$ with
\[
g\ge 2,\quad
\lambda\ge0,\quad
0\le\delta\le\lambda,\quad
j\ge 0,\quad
k\ge1,\quad
kg^{\delta}\mid g^{\lambda},\quad
g\nmid k,
\]
we have
\[
\sum_{\substack{
0\le h<g^{\lambda}\\
h\equiv a\ \mod{kg^{\delta}}
}}
\biggl|F_{\lambda}^{[j]}\biggl(\frac{h+\beta}{g^{\lambda}}\biggr)\biggr|
\le
g\biggl(\frac{g^{\lambda}}{kg^{\delta}}\biggr)^{\eta_{g}}
\biggl|F_{\delta}^{[j+\lambda-\delta]}\biggl(\frac{a+\beta}{g^{\delta}}\biggr)\biggr|.
\]
In particular, for $k=1$, $\delta=0$, we have
\begin{equation}
\label{lem:F_L1_moment:pure}
\sum_{0\le h<g^{\lambda}}
\biggl|F_{\lambda}^{[j]}\biggl(\frac{h+\beta}{g^{\lambda}}\biggr)\biggr|
\le
g^{\eta_{g}\lambda+1}.
\end{equation}
\end{lemma}
\begin{proof}
When $\lambda=\delta$, then the condition $kg^{\delta}\mid g^{\lambda}$ implies
$k=1$ and so the left-hand side of the assertion is $|F_{\delta}^{[j]}(\frac{a+\beta}{g^{\delta}})|$
and the right-hand side of the assertion is $g|F_{\delta}^{[j]}(\frac{a+\beta}{g^{\delta}})|$
so that the assertion holds trivially.
We may thus assume $\lambda>\delta$.

For integers $\theta$ with $\delta\le\theta\le\lambda$, we write
\begin{equation}
\label{lem:F_L1_moment:def:d_u}
d_{\theta}\coloneqq(g^{\theta},kg^{\delta})
\quad\text{and}\quad
u_{\theta}\coloneqq\frac{g^{\theta}}{d_{\theta}}.
\end{equation}
For $\delta<\theta\le\lambda$, since
$d_{\theta-1}=(g^{\theta-1},kg^{\delta})\mid(g^{\theta},kg^{\delta})=d_{\theta}$,
we also write
\begin{equation}
\label{lem:F_L1_moment:def:rho}
\rho_{\theta}
\coloneqq
\frac{d_{\theta}}{d_{\theta-1}}\in\mathbb{N}.
\end{equation}
Note that we then have
\[
d_{\theta-1}(\rho_{\theta},g)
=
(d_{\theta},gd_{\theta-1})
=
(g^{\theta},kg^{\delta},g^{\theta},kg^{\delta+1})
=
(g^{\theta},kg^{\delta})
=
d_{\theta}
=
\rho_{\theta}d_{\theta-1}
\]
so that
\begin{equation}
\label{lem:F_L1_moment:rho_div_g}
\rho_{\theta}\mid g.
\end{equation}
Also, we have
\[
\rho_{\theta}=g
\implies
g^{\delta+1}
\mid
g\cdot g^{\delta}(g^{\theta-\delta-1},k)
=
gd_{\theta-1}
=
\rho_{\theta}d_{\theta-1}
=
d_{\theta}
=
g^{\delta}(g^{\theta-\delta},k)
\]
and so $g\mid (g^{\theta-\delta},k)\mid k$,
which contradicts the assumption $g\nmid k$. Thus, we have
\begin{equation}
\label{lem:F_L1_moment:rho_less_g}
\rho_{\theta}<g.
\end{equation}
Also, we have
\[
d_{\theta-1}(\rho_{\theta},u_{\theta-1})
=
(d_{\theta},g^{\theta-1})
=
(g^{\theta},kg^{\delta},g^{\theta-1})
=
(g^{\theta-1},kg^{\delta})
=
d_{\theta-1}
\]
and so
\begin{equation}
\label{lem:F_L1_moment:rho_u_coprime}
(\rho_{\theta},u_{\theta-1})=1.
\end{equation}

We first prove a reduction formula for
\[
G_{\theta}^{[j]}(a,d_{\theta})
\coloneqq
\sum_{\substack{
0\le h<g^{\theta}\\
h\equiv a\ \mod{d_{\theta}}
}}
\biggl|F_{\theta}^{[j]}\biggl(\frac{h+\beta}{g^{\theta}}\biggr)\biggr|.
\]
Assume $\delta<\theta\le\lambda$. We can write
\[
G_{\theta}^{[j]}(a,d_{\theta})
=
\sum_{0\le u<u_{\theta}}
\biggl|F_{\theta}^{[j]}\biggl(\frac{a+ud_{\theta}+\beta}{g^{\theta}}\biggr)\biggr|.
\]
Note that
\begin{equation}
\label{lem:F_L1_moment:u_div_u}
\frac{u_{\theta}}{u_{\theta-1}}
=
\frac{g^{\theta}}{g^{\theta-1}}
\frac{d_{\theta-1}}{d_{\theta}}
=
\frac{g}{\rho_{\theta}}
\in\mathbb{N}
\end{equation}
since $\rho_{\theta}\mid g$ as in \cref{lem:F_L1_moment:rho_div_g}. We can thus change variables via $u=su_{\theta-1}+v$ as
\[
G_{\theta}^{[j]}(a,d_{\theta})
=
\sum_{0\le v<u_{\theta-1}}
\sum_{0\le s<g/\rho_{\theta}}
\biggl|F_{\theta}^{[j]}\biggl(
\frac{a+vd_{\theta}+su_{\theta-1}d_{\theta}+\beta}{g^{\theta}}
\biggr)\biggr|.
\]
Since $u_{\theta-1}d_{\theta}=g^{\theta-1}\rho_{\theta}$, we get
\begin{equation}
\label{lem:L1_moment:first_step}
G_{\theta}^{[j]}(a,d_{\theta})
=
\sum_{0\le v<u_{\theta-1}}
\sum_{0\le s<g/\rho_{\theta}}
\biggl|F_{\theta}^{[j]}\biggl(
\frac{a+vd_{\theta}+sg^{\theta-1}\rho_{\theta}+\beta}{g^{\theta}}
\biggr)\biggr|.
\end{equation}
By using \cref{lem:F_recursion}, periodicity, and a trivial bound $\phi_{j}(\beta)\le g$, we have
\begin{align}
&\biggl|F_{\theta}^{[j]}\biggl(
\frac{a+vd_{\theta}+sg^{\theta-1}\rho_{\theta}+\beta}{g^{\theta}}
\biggr)\biggr|\\
&=
\biggl|F_{\theta-1}^{[j+1]}\biggl(
\frac{a+vd_{\theta}+sg^{\theta-1}\rho_{\theta}+\beta}{g^{\theta-1}}
\biggr)\biggr|
\times
\frac{1}{g}
\phi_{j}\biggl(
\frac{a+vd_{\theta}+sg^{\theta-1}\rho_{\theta}+\beta}{g^{\theta}}
\biggr)\\
&\le
\biggl|F_{\theta-1}^{[j+1]}\biggl(
\frac{a+vd_{\theta}+\beta}{g^{\theta-1}}
\biggr)\biggr|.
\end{align}
In particular, we have the bound
\begin{align}
G_{\theta}^{[j]}(a,d_{\theta})
&\le
\frac{g}{\rho_{\theta}}
\sum_{0\le v<u_{\theta-1}}
\biggl|F_{\theta-1}^{[j+1]}\biggl(
\frac{a+vd_{\theta}+\beta}{g^{\theta-1}}
\biggr)\biggr|
=
\frac{g}{\rho_{\theta}}
\sum_{0\le v<u_{\theta-1}}
\biggl|F_{\theta-1}^{[j+1]}\biggl(
\frac{a+v\rho_{\theta}d_{\theta-1}+\beta}{g^{\theta-1}}
\biggr)\biggr|.
\end{align}
Since $(\rho_{\theta},u_{\theta-1})=1$ and the function
\[
\mathbb{Z}\to\mathbb{R}
\semicolon
v
\mapsto
\biggl|F_{\theta-1}^{[j+1]}\biggl(
\frac{a+vd_{\theta-1}+\beta}{g^{\theta-1}}
\biggr)\biggr|
\]
has a period $u_{\theta-1}$, we can change variables to obtain
\begin{equation}
\label{lem:F_L1_moment:one_step}
G_{\theta}^{[j]}(a,d_{\theta})
\le
\frac{g}{\rho_{\theta}}
\sum_{0\le v<u_{\theta-1}}
\biggl|F_{\theta-1}^{[j+1]}\biggl(
\frac{a+vd_{\theta-1}+\beta}{g^{\theta-1}}
\biggr)\biggr|
=
\frac{g}{\rho_{\theta}}
G_{\theta-1}^{[j+1]}(a,d_{\theta}).
\end{equation}
This is a trivial one-step reduction.

We next consider two-step reduction assuming $\delta+1<\theta\le\lambda$.
We change variables via
\[
v=ru_{\theta-2}+u
\]
in \cref{lem:L1_moment:first_step}. This gives
\[
G_{\theta}^{[j]}(a,d_{\theta})
=
\sum_{0\le u<u_{\theta-2}}
\sum_{0\le r<g/\rho_{\theta-1}}
\sum_{0\le s<g/\rho_{\theta}}
\biggl|F_{\theta}^{[j]}\biggl(
\frac{a+ud_{\theta}
+ru_{\theta-2}d_{\theta}
+sg^{\theta-1}\rho_{\theta}
+\beta}{g^{\theta}}
\biggr)\biggr|.
\]
Since $u_{\theta-2}d_{\theta}=g^{\theta-2}\rho_{\theta-1}\rho_{\theta}$,
by writing
\[
R_{\theta}
\coloneqq
\frac{g}{\rho_{\theta}},
\]
we get
\begin{align}
G_{\theta}^{[j]}(a,d_{\theta})
&=
\sum_{0\le u<u_{\theta-2}}
\sum_{0\le r<g/\rho_{\theta-1}}
\sum_{0\le s<g/\rho_{\theta}}
\biggl|F_{\theta}^{[j]}\biggl(
\frac{a+ud_{\theta}
+rg^{\theta-2}\rho_{\theta}\rho_{\theta-1}
+sg^{\theta-1}\rho_{\theta}
+\beta}{g^{\theta}}
\biggr)\biggr|\\
&=
\sum_{0\le u<u_{\theta-2}}
\sum_{0\le r<R_{\theta-1}}
\sum_{0\le s<R_{\theta}}
\biggl|F_{\theta}^{[j]}\biggl(
\frac{a+ud_{\theta}+\beta}{g^{\theta}}
+\frac{r}{R_{\theta-1}R_{\theta}}
+\frac{s}{R_{\theta}}
\biggr)\biggr|.
\end{align}
By applying \cref{lem:F_recursion} twice and by using periodicity and $R_{\theta}\mid g$, we have
\begin{align}
&\biggl|F_{\theta}^{[j]}\biggl(
\frac{a+ud_{\theta}+\beta}{g^{\theta}}
+\frac{r}{R_{\theta-1}R_{\theta}}
+\frac{s}{R_{\theta}}
\biggr)\biggr|\\
&=\biggl|F_{\theta-1}^{[j+1]}\biggl(
\frac{a+ud_{\theta}+\beta}{g^{\theta-1}}
+\frac{gr}{R_{\theta-1}R_{\theta}}
+\frac{gs}{R_{\theta}}
\biggr)\biggr|\\
&\hspace{0.1\textwidth}
\times
\frac{1}{g}\phi_{j}\biggl(
\frac{a+ud_{\theta}+\beta}{g^{\theta}}
+\frac{r}{R_{\theta-1}R_{\theta}}
+\frac{s}{R_{\theta}}
\biggr)\\
&=
\biggl|F_{\theta-1}^{[j+1]}\biggl(
\frac{a+ud_{\theta}+\beta}{g^{\theta-1}}
+\frac{gr}{R_{\theta-1}R_{\theta}}
\biggr)\biggr|\\
&\hspace{0.1\textwidth}
\times
\frac{1}{g}\phi_{j}\biggl(
\frac{a+ud_{\theta}+\beta}{g^{\theta}}
+\frac{r}{R_{\theta-1}R_{\theta}}
+\frac{s}{R_{\theta}}
\biggr)\\
&=
\biggl|F_{\theta-2}^{[j+2]}\biggl(
\frac{a+ud_{\theta}+\beta}{g^{\theta-2}}
+\frac{g^{2}r}{R_{\theta-1}R_{\theta}}
\biggr)\biggr|\\
&\hspace{0.1\textwidth}
\times
\frac{1}{g^{2}}
\phi_{j+1}\biggl(
\frac{a+ud_{\theta}+\beta}{g^{\theta-1}}
+\frac{gr}{R_{\theta-1}R_{\theta}}
\biggr)
\phi_{j}\biggl(
\frac{a+ud_{\theta}+\beta}{g^{\theta}}
+\frac{r}{R_{\theta-1}R_{\theta}}
+\frac{s}{R_{\theta}}
\biggr)\\
&=
\biggl|F_{\theta-2}^{[j+2]}\biggl(
\frac{a+ud_{\theta}+\beta}{g^{\theta-2}}
\biggr)\biggr|\\
&\hspace{0.1\textwidth}
\times
\frac{1}{g^{2}}
\phi_{j+1}\biggl(
\frac{a+ud_{\theta}+\beta}{g^{\theta-1}}
+\frac{gr}{R_{\theta-1}R_{\theta}}
\biggr)
\phi_{j}\biggl(
\frac{a+ud_{\theta}+\beta}{g^{\theta}}
+\frac{r}{R_{\theta-1}R_{\theta}}
+\frac{s}{R_{\theta}}
\biggr).
\end{align}
We thus have
\begin{equation}
\label{lem:F_L1_moment:two_step:prefinal}
\begin{aligned}
&G_{\theta}^{[j]}(a,d_{\theta})\\
&\le
\sum_{0\le u<u_{\theta-2}}
\biggl|F_{\theta-2}^{[j+2]}\biggl(
\frac{a+ud_{\theta}+\beta}{g^{\theta-2}}
\biggr)\biggr|\\
&\hspace{0.1\textwidth}
\times
\frac{1}{g^{2}}
\sum_{0\le r<R_{\theta-1}}
\sum_{0\le s<R_{\theta}}
\phi_{j+1}\biggl(
\frac{a+ud_{\theta}+\beta}{g^{\theta-1}}
+\frac{gr}{R_{\theta-1}R_{\theta}}
\biggr)\\
&\hspace{0.4\textwidth}
\times
\phi_{j}\biggl(
\frac{a+ud_{\theta}+\beta}{g^{\theta}}
+\frac{r}{R_{\theta-1}R_{\theta}}
+\frac{s}{R_{\theta}}
\biggr)\\
&=
\sum_{0\le u<u_{\theta-2}}
\biggl|F_{\theta-2}^{[j+2]}\biggl(
\frac{a+ud_{\theta}+\beta}{g^{\theta-2}}
\biggr)\biggr|
\Psi_{j}(t_{\theta},R_{\theta-1},R_{\theta})
\end{aligned}
\end{equation}
with
\[
t_{\theta}
\coloneqq
\frac{(a+ud_{\theta}+\beta)R_{\theta-1}R_{\theta}}{g^{\theta}}.
\]
We then use \cref{lem:Psi_bound} with
$i\coloneqq j$, $t\coloneqq t_{\theta}$,
$R\coloneqq R_{\theta-1}$ and $S\coloneqq R_{\theta}$.
We then have $R\mid g$ and $S\mid g$.
By \cref{lem:F_L1_moment:rho_less_g} and $\delta+1<\theta\le\lambda$,
we have $R\ge 2$. Also, since
\[
(gd_{\theta-2},d_{\theta})
=
(g^{\theta-1},kg^{\delta+1},g^{\theta},kg^{\delta})
=
(g^{\theta-1},kg^{\delta})
=
d_{\theta-1},
\]
we have
\[
(R,g/S)
=(R_{\theta-1},g/R_{\theta})
=(g/\rho_{\theta-1},\rho_{\theta})
=(gd_{\theta-2},d_{\theta})/d_{\theta-1}
=
1.
\]
Thus, by using \cref{lem:Psi_bound} in \cref{lem:F_L1_moment:two_step:prefinal}, we have
\begin{align}
G_{\theta}^{[j]}(a,d_{\theta})
&\le
(R_{\theta-1}R_{\theta})^{\eta_{g}}
\sum_{0\le u<u_{\theta-2}}
\biggl|F_{\theta-2}^{[j+2]}\biggl(
\frac{a+ud_{\theta}+\beta}{g^{\theta-2}}
\biggr)\biggr|\\
&=
\biggl(\frac{g^{2}}{\rho_{\theta-1}\rho_{\theta}}\biggr)^{\eta_{g}}
\sum_{0\le u<u_{\theta-2}}
\biggl|F_{\theta-2}^{[j+2]}\biggl(
\frac{a+ud_{\theta-2}\rho_{\theta-1}\rho_{\theta}+\beta}{g^{\theta-2}}
\biggr)\biggr|
\end{align}
Since \cref{lem:F_L1_moment:rho_u_coprime}
and \cref{lem:F_L1_moment:u_div_u}
imply $(\rho_{\theta}\rho_{\theta-1},u_{\theta-2})=1$
and the function
\[
\mathbb{Z}\to\mathbb{R}
\semicolon
u
\mapsto
\biggl|F_{\theta-2}^{[j+2]}\biggl(
\frac{a+ud_{\theta-2}+\beta}{g^{\theta-2}}
\biggr)\biggr|
\]
has a period $u_{\theta-2}$, we can change variables to get
\begin{equation}
\label{lem:F_L1_moment:two_step}
G_{\theta}^{[j]}(a,d_{\theta})
\le
\biggl(\frac{g^{2}}{\rho_{\theta-1}\rho_{\theta}}\biggr)^{\eta_{g}}
\sum_{0\le u<u_{\theta-2}}
\biggl|F_{\theta-2}^{[j+2]}\biggl(
\frac{a+ud_{\theta-2}+\beta}{g^{\theta-2}}
\biggr)\biggr|
=
\biggl(\frac{g^{2}}{\rho_{\theta-1}\rho_{\theta}}\biggr)^{\eta_{g}}
G_{\theta-2}^{[j+2]}(a,d_{\theta}).
\end{equation}
This is a two-step reduction.

Finally, by using \cref{lem:F_L1_moment:two_step} $[(\lambda-\delta)/2]$-times
and by using \cref{lem:F_L1_moment:one_step} if $\lambda-\delta$ is odd, we obtain
\begin{align}
\sum_{\substack{
0\le h<g^{\lambda}\\
h\equiv a\ \mod{kg^{\delta}}
}}
\biggl|F_{\lambda}^{[j]}\biggl(\frac{h+\beta}{g^{\lambda}}\biggr)\biggr|
&=
G_{\lambda}^{[j]}(a,d_{\lambda})
\le
\cdots
\le
(\rho_{\lambda}\cdots\rho_{\delta+1})^{-\eta_{g}}
g^{2[(\lambda-\delta)/2]\eta_{g}+1}
G_{\delta}^{[j+\lambda-\delta]}(a,d_{\delta}).
\end{align}
since $0<\eta_{g}<1$ by \cref{eta_bound}.
Since
$\rho_{\lambda}\cdots\rho_{\delta+1}
=
\frac{d_{\lambda}}{d_{\delta}}
=
\frac{kg^{\delta}}{g^{\delta}}
=
k$
and
\[
G_{\delta}^{[j+\lambda-\delta]}(a,d_{\delta})
=
G_{\delta}^{[j+\lambda-\delta]}(a,g^{\delta})
=
\sum_{\substack{
0\le h<g^{\delta}\\
h\equiv a\ \mod{g^{\delta}}
}}
\biggl|
F_{\delta}^{[j+\lambda-\delta]}\biggl(\frac{h+\beta}{g^{\delta}}\biggr)
\biggr|
=
\biggl|
F_{\delta}^{[j+\lambda-\delta]}\biggl(\frac{a+\beta}{g^{\delta}}\biggr)
\biggr|,
\]
we obtain the assertion.
\end{proof}

\subsection{A hybrid bound}
\label{subsec:large_sieve}
We finally prove a hybrid bound by combining the $L^{\infty}$ and $L^{1}$ bounds.

\begin{lemma}
\label{lem:continous_L1}
For $g\in\mathbb{Z}_{\ge2}$, $\bm{\alpha}\in\mathscr{A}_{g}$, $\lambda,j\in\mathbb{Z}_{\ge0}$, we have
\[
\int_{0}^{1}
|F_{\lambda}^{[j]}(\beta)|d\beta
\ll
g^{-(1-\eta_{g})\lambda},
\]
where the implicit constant depends only on $g$.
\end{lemma}
\begin{proof}
By decomposing the integral, we have
\begin{align}
\int_{0}^{1}
|F_{\lambda}^{[j]}(\beta)|d\beta
&=
\sum_{0\le h<g^{\lambda}}
\int_{\frac{h}{g^{\lambda}}}^{\frac{h+1}{g^{\lambda}}}
|F_{\lambda}^{[j]}(\beta)|d\beta\\
&=
\sum_{0\le h<g^{\lambda}}
\int_{0}^{\frac{1}{g^{\lambda}}}
\biggl|F_{\lambda}^{[j]}\biggl(\frac{h+\beta g^{\lambda}}{g^{\lambda}}\biggr)\biggr|d\beta
=
\int_{0}^{\frac{1}{g^{\lambda}}}
\biggl(
\sum_{0\le h<g^{\lambda}}
\biggl|F_{\lambda}^{[j]}\biggl(\frac{h+\beta g^{\lambda}}{g^{\lambda}}\biggr)\biggr|
\biggr)d\beta.
\end{align}
Then, \cref{lem:F_L1_moment} gives the result.
\end{proof}

\begin{lemma}
\label{lem:continous_L1_deriv}
For $g\in\mathbb{Z}_{\ge2}$, $\bm{\alpha}\in\mathscr{A}_{g}$, $\lambda,j\in\mathbb{Z}_{\ge0}$, we have
\[
\int_{0}^{1}
\biggl|\frac{\partial}{\partial\beta}F_{\lambda}^{[j]}(\beta)\biggr|d\beta
\ll
g^{\eta_{g}\lambda},
\]
where the implicit constant is absolute.
\end{lemma}
\begin{proof}
Note that we have
\[
\frac{\partial}{\partial\beta}F_{\lambda}^{[j]}(\beta)
=
-\frac{2\pi i}{g^{\lambda}}
\sum_{0\le n<g^{\lambda}}
ne(f_{\lambda}^{[j]}(n)-\beta n).
\]
By using the partial summation and \cref{lem:sum_cleanup}, we get
\begin{align}
\biggl|\frac{\partial}{\partial\beta}F_{\lambda}^{[j]}(\beta)\biggr|
&\ll
\sup_{0\le x\le g^{\lambda}}
\biggl|
\sum_{0\le n<x}
e(f_{\lambda}^{[j]}(n)-\beta n)
\biggr|
\ll
\sum_{0\le\ell\le\lambda}
g^{\ell}|F_{\ell}^{[j]}(\beta)|
\end{align}
Thus, by using \cref{lem:continous_L1}, we get
\[
\int_{0}^{1}
\biggl|\frac{\partial}{\partial\beta}F_{\lambda}^{[j]}(\beta)\biggr|d\beta
=
\sum_{0\le\ell\le\lambda}
g^{\ell}
\int_{0}^{1}
|F_{\ell}^{[j]}(\beta)|
d\beta
\ll
\sum_{0\le\ell\le\lambda}
g^{\eta_{g}\ell}
\ll
g^{\eta_{g}\lambda},
\]
where we used $\eta_{g}>0.2075187$ in the last step.
This gives the assertion.
\end{proof}

\begin{lemma}[Galllagher--Sobolev inequality]
\label{lem:Gallagher_Sobolev}
Let $f\colon[0,1]\to\mathbb{C}$ be a function of class $C^{1}$ of period $1$,
$\delta>0$ and $(\alpha_{i})_{i=1}^{R}$ be a sequence of real numbers
which is $\delta$-spaced $\mod{1}$, i.e.\ 
\[
\|\alpha_{i}-\alpha_{j}\|\ge\delta
\quad\text{for any}\quad
i,j\in\{1,\ldots,R\}
\ \text{with}\ 
i\neq j.
\]
We then have
\[
\sum_{i=1}^{R}|f(\alpha_{i})|
\le
\frac{1}{\delta}\int_{0}^{1}|f(\alpha)|d\alpha
+
\frac{1}{2}
\int_{0}^{1}|f'(\alpha)|d\alpha.
\]
\end{lemma}
\begin{proof}
See Lemma~1.2 of \cite[p.~2]{Montgomery:Topics}.
\end{proof}

\begin{lemma}
\label{lem:hybrid_bound}
For $g\in\mathbb{Z}_{\ge2}$, $\bm{\alpha}\in\mathscr{A}_{g}$, $\lambda,j\in\mathbb{Z}_{\ge0}$ and $M\ge1$,
by writing $\mu\coloneqq[\frac{\log M}{\log g}]$,
we have
\[
\sum_{M\le m\le 2M}
\sum_{\substack{
0\le k<m\\
(k,m)=1
}}
\biggl|F_{\lambda}^{[j]}\biggl(\frac{k}{m}\biggr)\biggr|
\ll
\left\{
\begin{array}{>{\displaystyle}cl}
Mg^{-(\frac{1}{2}-\eta_{g})2\mu-\sigma_{\lambda-2\mu}^{[j+2\mu]}(\bm{\alpha})}&\text{if $2\mu\le\lambda$},\\[2mm]
M^{2}g^{-(1-\eta_{g})\lambda}&\text{if $2\mu>\lambda$},
\end{array}
\right.
\]
where the implicit constant depends only on $g$.
\end{lemma}
\begin{proof}
Take a parameter $\lambda_{1}\in[0,\lambda]\cap\mathbb{Z}$ chosen later.
By \cref{lem:F_product_formula}, we have
\begin{align}
\biggl|F_{\lambda}^{[j]}\biggl(\frac{k}{m}\biggr)\biggr|
&=
\frac{1}{g^{\lambda_{1}}}
\prod_{0\le i<\lambda_{1}}
\phi_{i}^{[j]}\biggl(\frac{k}{m}g^{i}\biggr)
\times
\frac{1}{g^{\lambda-\lambda_{1}}}
\prod_{\lambda_{1}\le i<\lambda}
\phi_{i}^{[j]}\biggl(\frac{k}{m}g^{i}\biggr)\\
&=
\frac{1}{g^{\lambda_{1}}}
\prod_{0\le i<\lambda_{1}}
\phi_{i}^{[j]}\biggl(\frac{k}{m}g^{i}\biggr)
\times
\frac{1}{g^{\lambda-\lambda_{1}}}
\prod_{0\le i<\lambda-\lambda_{1}}
\phi_{i}^{[j+\lambda_{1}]}\biggl(\frac{k}{m}g^{\lambda_{1}}\cdot g^{i}\biggr)\\
&=
\biggl|F_{\lambda_{1}}^{[j]}\biggl(\frac{k}{m}\biggr)\biggr|
\biggl|F_{\lambda-\lambda_{1}}^{[j+\lambda_{1}]}\biggl(\frac{k}{m}g^{\lambda_{1}}\biggr)\biggr|.
\end{align}
By using \cref{lem:L_infty_general}, we thus have
\[
S
\coloneqq
\sum_{M\le m\le 2M}
\sum_{\substack{
0\le k<m\\
(k,m)=1
}}
\biggl|F_{\lambda}^{[j]}\biggl(\frac{k}{m}\biggr)\biggr|
\ll
g^{-\sigma_{\lambda-\lambda_{1}}^{[j+\lambda_{1}]}(\bm{\alpha})}
\sum_{M\le m\le 2M}
\sum_{\substack{
0\le k<m\\
(k,m)=1
}}
\biggl|F_{\lambda_{1}}^{[j]}\biggl(\frac{k}{m}\biggr)\biggr|.
\]
Since the reduced fractions
\[
\frac{k}{m}
\quad\text{for}\quad
k,m\in\mathbb{Z}
\ \text{with}\ 
M<m\le 2M,\ 
0\le k<m,\ 
(k,m)=1
\]
are $(2M)^{-2}$-spaced, the Gallagher--Sobolev inequality (\cref{lem:Gallagher_Sobolev}) gives
\begin{align}
S
\ll
g^{-\sigma_{\lambda-\lambda_{1}}^{[j+\lambda_{1}]}(\bm{\alpha})}
\biggl(
M^{2}
\int_{0}^{1}|F_{\lambda_{1}}^{[j]}(\beta)|d\beta
+
\int_{0}^{1}
\biggl|\frac{\partial}{\partial\beta}F_{\lambda_{1}}^{[j]}(\beta)\biggr|d\beta
\biggr).
\end{align}
By using \cref{lem:continous_L1} and \cref{lem:continous_L1_deriv}, we further have
\begin{equation}
\label{lem:hybrid_bound:prefinal}
S
\ll
g^{\eta_{g}\lambda_{1}-\sigma_{\lambda-\lambda_{1}}^{[j+\lambda_{1}]}(\bm{\alpha})}
(M^{2}g^{-\lambda_{1}}+1).
\end{equation}
When $2\mu\le\lambda$, by taking $\lambda_{1}$ by $\lambda_{1}\coloneqq2\mu\in[0,\lambda]\cap\mathbb{Z}$
in \cref{lem:hybrid_bound:prefinal} so that
\[
g^{\lambda_{1}}\asymp M^{2}
\quad\text{and}\quad
g^{\eta_{g}\lambda_{1}}
=
g^{\frac{1}{2}\lambda_{1}}
\cdot
g^{-(\frac{1}{2}-\eta_{g})\lambda_{1}}
\asymp
M
g^{-(\frac{1}{2}-\eta_{g})2\mu},
\]
we obtain
\begin{equation}
\label{lem:hybrid_bound:small_mu}
S
\ll
M
g^{-(\frac{1}{2}-\eta_{g})2\mu-\sigma_{\lambda-2\mu}^{[j+2\mu]}(\bm{\alpha})}.
\end{equation}
When $\lambda<2\mu$, we have $M^{2}g^{-\lambda}\gg1$
and so, by taking $\lambda_{1}$ by $\lambda_{1}\coloneqq\lambda$
in \cref{lem:hybrid_bound:prefinal}, we obtain
\begin{equation}
\label{lem:hybrid_bound:large_mu}
S
\ll
M^{2}
g^{-(1-\eta_{g})\lambda}.
\end{equation}
By combining \cref{lem:hybrid_bound:small_mu} and \cref{lem:hybrid_bound:large_mu},
we obtain the result.
\end{proof}

\section{The Type I sum estimate}
\label{sec:TypeI}
As usual, we need to estimate bilinear sums called the Type I and Type II sums.
In this section, we prepare the Type I sum estimate for weakly digital functions.

We start with a small lemma on $\sigma_{\lambda}^{[j]}(\bm{\alpha})$.
\begin{lemma}
\label{lem:sigma_monotonicity}
Let $g\in\mathbb{Z}_{\ge2}$, $\bm{\alpha}\in\mathscr{A}_{g}$.
\begin{enumerate}[]
\item\label{lem:sigma_monotonicity:one_shift}
For $\lambda\in\mathbb{N}$, $j\in\mathbb{Z}_{\ge0}$, we have
\[
\sigma_{\lambda}^{[j]}(\bm{\alpha})-\frac{\log 2}{4g^{3}(\log g)^{2}}
\le
\sigma_{\lambda-1}^{[j+1]}(\bm{\alpha})
\le
\sigma_{\lambda}^{[j]}(\bm{\alpha}).
\]
\item\label{lem:sigma_monotonicity:with_eta}
For $\lambda,j\in\mathbb{Z}_{\ge0}$ and $A\ge\frac{\log 2}{\log g}$, the function
\[
f\colon
[0,\lambda]\cap\mathbb{Z}
\to
\mathbb{R}
\semicolon
\mu
\mapsto
A\biggl(\frac{1}{2}-\eta_{g}\biggr)\mu+\sigma_{\lambda-\mu}^{[j+\mu]}(\bm{\alpha})
\]
is increasing.
\item\label{lem:sigma_monotonicity:with_lambda}
For $A\in\mathbb{R}$ with $A\ge\frac{\log 2}{4g^{3}(\log g)^{2}}$ and $j\in\mathbb{Z}_{\ge0}$, the function
\[
f\colon
\mathbb{Z}_{\ge0}
\to
\mathbb{R}
\semicolon
\lambda\mapsto A\lambda-\sigma_{\lambda}^{[j]}(\bm{\alpha})
\]
is increasing.
\end{enumerate}
\end{lemma}
\begin{proof}
\leavevmode\par\medskip

\proofitem{lem:sigma_monotonicity:one_shift}
By the definition \cref{def:sigma_j_alpha} of $\sigma_{\lambda}^{[j]}(\bm{\alpha})$
and \cref{gamma_bound}, we have
\[
\sigma_{\lambda-1}^{[j+1]}(\bm{\alpha})
=
\sum_{0\le i<\lambda-1}
\gamma_{i}^{[j+1]}(\bm{\alpha})
=
\sum_{0\le i<\lambda-1}
\gamma_{i+1}^{[j]}(\bm{\alpha})
=
\sum_{1\le i<\lambda}
\gamma_{i}^{[j]}(\bm{\alpha})
=
\sigma_{\lambda}^{[j]}(\bm{\alpha})
-
\gamma_{0}^{[j]}(\bm{\alpha}).
\]
By using \cref{gamma_bound}, we obtain \cref{lem:sigma_monotonicity:one_shift}.
\medskip

\proofitem{lem:sigma_monotonicity:with_eta}
By \cref{lem:sigma_monotonicity:one_shift} proven above,
for $\mu\in[0,\lambda-1]\cap\mathbb{Z}$, we have
\begin{align}
\label{lem:sigma_monotonicity:with_eta:diff}
f(\mu+1)-f(\mu)
&=
A\biggl(\frac{1}{2}-\eta_{g}\biggr)
+
\sigma_{\lambda-(\mu+1)}^{[j+\mu+1]}(\bm{\alpha})
-
\sigma_{\lambda-\mu}^{[j+\mu]}(\bm{\alpha})\\
&\ge
\frac{\log 2}{\log g}\biggl(\frac{1}{2}-\eta_{g}\biggr)
-
\frac{\log 2}{4g^{3}(\log g)^{2}}
>0
\end{align}
by \cref{eta_bound}.
This proves \cref{lem:sigma_monotonicity:with_eta}.
\medskip

\proofitem{lem:sigma_monotonicity:with_lambda}
By \cref{gamma_bound} and \cref{def:sigma_j_alpha}, we have
\[
f(\lambda+1)-f(\lambda)
=
A-\gamma_{\lambda}^{[j]}(\bm{\alpha})
>0,
\]
and so \cref{lem:sigma_monotonicity:with_lambda} follows.
\end{proof}

\begin{lemma}
\label{lem:TypeI}
For $g\in\mathbb{Z}_{\ge2}$, $\bm{\alpha}\in\mathscr{A}_{g}$,
$L\in\mathbb{N}$, $x,M\in\mathbb{R}$ with $2\le x\le g^{L}$, we have
\[
S_{\mathrm{I}}
\coloneqq
\sum_{1\le m\le M}
\sup_{1\le t\le x/m}
\biggl|
\sum_{1\le n\le t}
e(f_{L}(mn))
\biggr|
\ll
xg^{-\kappa_{\mathrm{I}}}(\log x)^{2}
\]
provided $M\le x^{\frac{1}{2}}$, where
\begin{equation}
\label{lem:TypeI:def:kappa}
\kappa_{\mathrm{I}}
\coloneqq
\sigma_{\xi_{\mathrm{I}}}(\bm{\alpha})
\quad\text{with}\quad
\xi_{\mathrm{I}}
\coloneqq
\biggl[\frac{\log x}{\log g}\biggr]
\end{equation}
and the implicit constant depends only on $g$.
\end{lemma}
\begin{proof}
By rewriting $mn$ to $n$, $mt$ to $t$ and taking care of the term $n=0$, we have
\[
S_{\mathrm{I}}
\le
\sum_{1\le m\le M}
\sup_{m\le t\le x}
\biggl|
\sum_{\substack{
0\le n\le t\\
n\equiv 0\ \mod{m}
}}
e(f_{L}(n))
\biggr|
+
M.
\]
By the orthogonality, we have
\begin{align}
S_{\mathrm{I}}
&\le
\sum_{1\le m\le M}
\sup_{m\le t\le x}
\biggl|
\frac{1}{m}
\sum_{0\le k<m}
\sum_{1\le n<t}
e\biggl(f_{L}(n)-\frac{kn}{m}\biggr)
\biggr|
+
M\\
\label{lem:TypeI:after_orthogonality}
&\le
\sum_{1\le m\le M}
\frac{1}{m}
\sum_{0\le k<m}
\sup_{m\le t\le x}
\biggl|
\sum_{1\le n<t}
e\biggl(f_{L}(n)-\frac{kn}{m}\biggr)
\biggr|
+
M.
\end{align}
We may replace $x$ by $[x]$ in \cref{lem:TypeI:after_orthogonality},
and then, by applying \cref{lem:sum_cleanup}, we can bound \cref{lem:TypeI:after_orthogonality} by
\begin{align}
S_{\mathrm{I}}
&\ll
\sum_{0\le\lambda\le\frac{\log x}{\log g}}
\sum_{1\le m\le M}
\frac{1}{m}
\sum_{0\le k<m}
\biggl|
\sum_{0\le n<g^{\lambda}}
e\biggl(f_{\lambda}(n)-\frac{kn}{m}\biggr)
\biggr|
+
M\\
&=
\sum_{0\le\lambda\le\frac{\log x}{\log g}}
\sum_{1\le m\le M}
\frac{g^{\lambda}}{m}
\sum_{0\le k<m}
\biggl|F_{\lambda}\biggl(\frac{k}{m}\biggr)\biggr|
+
M.
\end{align}
By classifying the terms according to the values of $(k,m)=d$, we have
\begin{equation}
\label{lem:TypeI:before_hybrid_bound}
\begin{aligned}
S_{\mathrm{I}}
&\ll
\sum_{0\le\lambda\le\frac{\log x}{\log g}}
\sum_{1\le d\le M}
\sum_{\substack{
1\le m\le M\\
d\mid m
}}
\frac{g^{\lambda}}{m}
\sum_{\substack{
0\le k<m\\
(k,m)=d
}}
\biggl|F_{\lambda}\biggl(\frac{k}{m}\biggr)\biggr|
+
M\\
&=
\sum_{0\le\lambda\le\frac{\log x}{\log g}}
\sum_{1\le d\le M}
\frac{1}{d}
\sum_{1\le m\le M/d}
\frac{g^{\lambda}}{m}
\sum_{\substack{
0\le k<m\\
(k,m)=1
}}
\biggl|F_{\lambda}\biggl(\frac{k}{m}\biggr)\biggr|
+
M.
\end{aligned}
\end{equation}

In order to bound the inner sum of \cref{lem:TypeI:before_hybrid_bound}, we bound the sum
\[
T
\coloneqq
\sum_{U\le m\le2U}
\frac{g^{\lambda}}{m}
\sum_{\substack{
0\le k<m\\
(k,m)=1}}
\biggl|F_{\lambda}\biggl(\frac{k}{m}\biggr)\biggr|
\]
for $1\le U\le M/d$ with $1\le d\le M$ by using \cref{lem:hybrid_bound}. Write
\[
\mu_{U}
\coloneqq
\biggl[\frac{\log U}{\log g}\biggr].
\]
When $2\mu_{U}\le\lambda$, \cref{lem:hybrid_bound} gives
\[
T
\ll
g^{\lambda}U^{-1}
\sum_{U\le m\le2U}
\sum_{\substack{
0\le k<m\\
(k,m)=1}}
\biggl|F_{\lambda}\biggl(\frac{k}{m}\biggr)\biggr|
\ll
g^{\lambda-(\frac{1}{2}-\eta_{g})2\mu_{U}-\sigma_{\lambda-2\mu_{U}}^{[j+2\mu_{U}]}(\bm{\alpha})}
\]
By \cref{lem:sigma_monotonicity:with_eta} and \cref{lem:sigma_monotonicity:with_lambda}
of \cref{lem:sigma_monotonicity}, this can be further bounded as
\begin{equation}
\label{lem:TypeI:T:large_lambda}
T
\ll
g^{\frac{1}{2}\lambda}g^{\frac{1}{2}\lambda-\sigma_{\lambda}(\bm{\alpha})}
\le
g^{\frac{1}{2}\lambda}g^{\frac{1}{2}\xi_{\mathrm{I}}-\sigma_{\xi_{\mathrm{I}}}(\bm{\alpha})}
\ll
g^{\frac{1}{2}\lambda}
x^{\frac{1}{2}}
g^{-\kappa_{\mathrm{I}}}.
\end{equation}
When $0\le\lambda<2\mu_{U}$, \cref{lem:hybrid_bound} gives
\begin{equation}
\label{lem:TypeI:T:small_lambda}
T
\ll
g^{\lambda}U^{-1}
\sum_{U\le m\le2U}
\sum_{\substack{
0\le k<m\\
(k,m)=1}}
\biggl|F_{\lambda}\biggl(\frac{k}{m}\biggr)\biggr|
\ll
Ug^{\eta_{g}\lambda}
\ll
\frac{M}{d}g^{\eta_{g}\lambda}.
\end{equation}
By \cref{lem:TypeI:T:large_lambda} and \cref{lem:TypeI:T:small_lambda}, in any case, we have
\begin{equation}
\label{lem:TypeI:T}
T
\ll
g^{\frac{1}{2}\lambda}
x^{\frac{1}{2}}
g^{-\kappa_{\mathrm{I}}}
+
\frac{M}{d}g^{\eta_{g}\lambda}.
\end{equation}

On inserting \cref{lem:TypeI:T} into \cref{lem:TypeI:before_hybrid_bound},
by recalling $\eta_{g}>0.2$ as seen in \cref{eta_bound}, we get
\begin{equation}
\label{lem:TypeI:prefinal}
S_{\mathrm{I}}
\ll
(\log x)^{2}
\sum_{0\le\lambda\le\frac{\log x}{\log g}}
(g^{\frac{1}{2}\lambda}
x^{\frac{1}{2}}
g^{-\kappa_{\mathrm{I}}}
+
Mg^{\eta_{g}\lambda})
+M
\ll
(xg^{-\kappa_{\mathrm{I}}}
+
Mx^{\eta_{g}})
(\log x)^{2}.
\end{equation}
By using \cref{sigma_bound} and \cref{eta_bound}, we have
\[
\kappa_{\mathrm{I}}
\le
\frac{\xi_{\mathrm{I}}}{4g^{3}\log g}
\le
\frac{\log x}{\log g}\biggl(\frac{1}{2}-\eta_{g}\biggr).
\]
Thus, the assumption $M\le x^{\frac{1}{2}}$ gives
\[
Mx^{\eta_{g}}
\le
x^{\frac{1}{2}+\eta_{g}}
=
x^{1-(\frac{1}{2}-\eta_{g})}
\le
xg^{-\kappa_{\mathrm{I}}}.
\]
On inserting this estimate into \cref{lem:TypeI:prefinal},
we obtain the result.
\end{proof}

\section{The Type II sum estimate}
\label{sec:TypeII}
For the Type-II sum estimate, we first prepare some known lemmas.

\begin{lemma}[Van der Corput's inequality]
\label{lem:vdC_ineq}
For a sequence of complex numbers $(z_{n})_{n=1}^{N}$ and $R\in\mathbb{N}$,
\[
\biggl|
\sum_{n=1}^{N}z_{n}
\biggr|^{2}
\le
\frac{N+R-1}{R}
\sum_{|r|<R}\biggl(1-\frac{|r|}{R}\biggr)
\sum_{1\le n,n+r\le N}z_{n+r}\overline{z_{n}}.
\]
\end{lemma}
\begin{proof}
See, e.g.\ Graham--Kolesnik~\cite[Lemma~2.5, p.~10]{GrahamKolesnik}.
\end{proof}

\begin{lemma}
\label{lem:truncation}
Let $g\in\mathbb{Z}_{\ge2}$, $r\in\mathbb{Z}_{\ge0}$, $M,N,R\in\mathbb{R}_{\ge1}$ with $0\le r\le R$.
Take $\lambda\in\mathbb{N}$ by
\begin{equation}
\label{lem:truncation:def_lambda}
g^{\lambda-1}\le MR^{2}<g^{\lambda}.
\end{equation}
Assume that
\begin{equation}
\label{lem:truncation:cond}
R\le N^{\frac{1}{2}}
\quad\text{and}\quad
\lambda\le L
\end{equation}
Consider the set
\begin{equation}
\label{lem:truncation:def:E}
\mathscr{E}
\coloneqq
\left\{
(m,n)\in\mathbb{N}^{2}
\midmid
\begin{gathered}
M<m\le 2M,\ N<n\le 2N\\
f_{L}(m(n+r))-f_{L}(mn)
\neq
f_{\lambda}(m(n+r))-f_{\lambda}(mn)
\end{gathered}
\right\}.
\end{equation}
Then, we have a bound
\begin{equation}
\label{lem:truncation:bound}
\#\mathscr{E}
\ll
\frac{MN}{R},
\end{equation}
where the implicit constant depends only on $g$.
\end{lemma}
\begin{proof}
For $(m,n)\in\mathbb{N}^{2}$, if there is no $k\in\mathbb{N}$ with
\begin{equation}
\label{lem:truncation:singular}
mn<kg^{\lambda}\le m(n+r),
\end{equation}
i.e.\ if $[\frac{mn}{g^{\lambda}}]=[\frac{m(n+r)}{g^{\lambda}}]$, then we have
\[
\epsilon_{i}(m(n+r))=\epsilon_{i}(mn)
\quad\text{for $i\ge\lambda$}
\]
since
\[
\sum_{i\ge\lambda}\epsilon_{i}(m(n+r))g^{i-\lambda}
=
\biggl[\frac{m(n+r)}{g^{\lambda}}\biggr]
=
\biggl[\frac{mn}{g^{\lambda}}\biggr]
=
\sum_{i\ge\lambda}\epsilon_{i}(mn)g^{i-\lambda}.
\]
Therefore, if there is no $k\in\mathbb{N}$ with \cref{lem:truncation:singular},
then since $\lambda\le L$, we have
\begin{align}
f_{L}(m(n+r))-f_{L}(mn)
&=
\sum_{0\le i<L}
\bigl(
\alpha_{i}(\epsilon_{i}(m(n+r)))
-
\alpha_{i}(\epsilon_{i}(mn))
\bigr)\\
&=
\sum_{0\le i<\lambda}
\bigl(
\alpha_{i}(\epsilon_{i}(m(n+r)))
-
\alpha_{i}(\epsilon_{i}(mn))
\bigr)
=
f_{\lambda}(m(n+r))-f_{\lambda}(mn).
\end{align}
Thus, if $(m,n)\in\mathscr{E}$, there exists $k\in\mathbb{N}$ satisfying \cref{lem:truncation:singular}.
Since such $k$ should satisfy
\[
mN<kg^{\lambda}\le m(2N+r),
\]
we have
\begin{equation}
\label{lem:truncation:prefinal}
\begin{aligned}
\#\mathscr{E}
&\le
\sum_{M<m\le 2M}
\sum_{\frac{mN}{g^{\lambda}}<k\le\frac{m(2N+r)}{g^{\lambda}}}
\sum_{\substack{
N<n\le 2N\\
mn<kg^{\lambda}\le m(n+r)
}}
1\\
&\le
\sum_{M<m\le 2M}
\sum_{\frac{mN}{g^{\lambda}}<k\le\frac{m(2N+r)}{g^{\lambda}}}
\sum_{\frac{kg^{\lambda}}{m}-r\le n<\frac{kg^{\lambda}}{m}}
1\\
&\ll
r\sum_{M<m\le 2M}
\biggl(\frac{m(N+r)}{g^{\lambda}}+1\biggr)
\ll
MR\biggl(\frac{M(N+R)}{g^{\lambda}}+1\biggr).
\end{aligned}
\end{equation}
By \cref{lem:truncation:cond} and \cref{lem:truncation:def_lambda}, we have $g^{\lambda}\asymp MR^{2}\ll MN$.
Thus, we can bound \cref{lem:truncation:prefinal} further as
\[
\#\mathscr{E}
\ll
\frac{M^{2}NR}{g^{\lambda}}
\ll
\frac{MN}{R}
\]
as desired. This completes the proof.
\end{proof}

\begin{lemma}
\label{lem:sin_sum}
For $a\in\mathbb{Z}$, $m\in\mathbb{N}$, $b,M\in\mathbb{R}$
with $d\coloneqq(a,m)$ and $M>0$, we have
\begin{align}
\label{lem:sin_sum:bound}
\sum_{0\le n<m}
\min\biggl(
M,
\frac{1}{|\sin\pi\frac{an+b}{m}|}
\biggr)
\le
d
\min\biggl(
M,
\frac{1}{\sin\pi\frac{d}{m}\|\frac{b}{d}\|}
\biggr)
+
\frac{d}{\sin\frac{\pi d}{2m}}
+
\frac{2m}{\pi}\log\frac{2m}{\pi d}.
\end{align}
\end{lemma}
\begin{proof}
See Lemme~6 of \cite[p.~1609]{MauduitRivat:Gelfond}.
\end{proof}

We can now prove the following Type II sum estimate.
\begin{lemma}
\label{lem:TypeII}
For $g\in\mathbb{Z}_{\ge2}$, $\bm{\alpha}\in\mathscr{A}_{g}$, $L\in\mathbb{N}$, $x,M,N,\theta\in\mathbb{R}$ with
$x\ge2$, $M,N\ge 1$, $\theta>0$, $1\le x\le g^{L}$
and complex coefficients $(a_{m})_{m\in\mathbb{N}}$, $(b_{n})_{n\in\mathbb{N}}$
with $|a_{m}|,|b_{n}|\le1$ for any $m$ and $n$, we have
\[
S_{\mathrm{II}}
\coloneqq
\sum_{\substack{
M<m\le 2M\\
N<n\le 2N\\
mn\le x}}
a_{m}b_{n}
e(f_{L}(mn))
\ll
xg^{-\kappa_{\mathrm{II}}}\log x,
\]
provided
\begin{equation}
\label{lem:TypeII:MN_range}
M,N\ge x^{\theta},
\end{equation}
where
\begin{equation}
\label{lem:TypeII:def:kappa}
\kappa_{\mathrm{II}}
\coloneqq
\frac{1}{10}
\sigma_{\xi_{\mathrm{II}}}(\bm{\alpha})
\quad\text{with}\quad
\xi_{\mathrm{II}}
\coloneqq
\biggl[\theta\cdot\frac{\log x}{\log g}\biggr]
\end{equation}
and $c_{\mathrm{II}}$ and the implicit constant depends only on $g$.
\end{lemma}
\begin{proof}
We may assume $MN<x$ since otherwise $S_{\mathrm{II}}$ is an empty sum
and the assertion is trivial.
By symmetry, we may assume $M\le N$.
Since $MN<x$, we then have $M\le x^{\frac{1}{2}}$.
Take $\mu\in\mathbb{N}$ with
\begin{equation}
\label{lem:TypeII:def:mu}
g^{\mu-1}\le M<g^{\mu}.
\end{equation}
By \cref{lem:TypeII:MN_range} and $M\le x^{\frac{1}{2}}$, we have
\begin{equation}
\label{lem:TypeII:mu_range}
\theta\cdot\frac{\log x}{\log g}<\mu\le\frac{1}{2}\frac{\log x}{\log g}+1.
\end{equation}
By the Cauchy--Schwarz inequality, we have
\begin{equation}
\label{lem:TypeII:CS}
|S_{\mathrm{II}}|^{2}
\le
M
\sum_{M<m\le 2M}
\biggl|
\sum_{\substack{
N<n\le 2N\\
mn\le x
}}
b_{n}
e(f_{L}(mn))
\biggr|^{2}.
\end{equation}
Take a real parameter
\begin{equation}
\label{lem:TypeII:R_range}
R\in[1,N^{\frac{1}{2}}].
\end{equation}
By applying \cref{lem:vdC_ineq} to \cref{lem:TypeII:CS} with $R$ replaced by $[R]+1$, we have
\begin{equation}
\label{lem:TypeII:vdC}
\begin{aligned}
|S_{\mathrm{II}}|^{2}
&\ll
\frac{M(N+[R])}{[R]+1}
\sum_{M<m\le 2M}
\sum_{|r|\le R}\biggl(1-\frac{|r|}{[R]+1}\biggr)\\
&\hspace{0.2\textwidth}
\times
\sum_{\substack{
N<n,n+r\le 2N\\
mn,m(n+r)\le x
}}
b_{n+r}\overline{b_{n}}
e\bigl(f_{L}(m(n+r))-f_{L}(mn)\bigr)\\
&\ll
\frac{MN}{R}
\sum_{|r|\le R}
\sum_{\substack{
N<n,n+r\le 2N
}}
\biggl|
\sum_{\substack{
M<m\le 2M\\
mn,m(n+r)\le x
}}
e\bigl(f_{L}(m(n+r))-f_{L}(mn)\bigr)
\biggr|.
\end{aligned}
\end{equation}
By putting the contribution of $r=0$ aside
and using the symmetry between $r$ and $-r$, we have
\begin{equation}
\label{lem:TypeII:vdC_clean}
|S_{\mathrm{II}}|^{2}
\ll
\frac{MN}{R}
\sum_{1\le r\le R}
\sum_{\substack{
N<n\le 2N
}}
\biggl|
\sum_{\substack{
M<m\le 2M\\
m(n+r)\le x
}}
e\bigl(f_{L}(m(n+r))-f_{L}(mn)\bigr)
\biggr|
+
\frac{M^{2}N^{2}}{R}.
\end{equation}
We now use \cref{lem:truncation} to truncate $f_{L}$.
By \cref{lem:TypeII:R_range}, the condition $R\le N^{\frac{1}{2}}$ is satisfied.
Take $\lambda\in\mathbb{N}$ by
\begin{equation}
\label{lem:TypeII:def:lambda}
g^{\lambda-1}\le MR^{2}<g^{\lambda}
\end{equation}
as in \cref{lem:truncation}.
Since we are assuming $MN<x$, by using $x\le g^{L}$ and $R\le N^{\frac{1}{2}}$, we have
\[
g^{\lambda-1}\le MR^{2}\le MN<x\le g^{L}
\]
and so the condition $\lambda\le L$ is also satisfied.
Thus, we can use \cref{lem:truncation} to truncate $f_{L}$ as
\begin{equation}
\label{lem:TypeII:truncation}
\begin{aligned}
|S_{\mathrm{II}}|^{2}
&\ll
\frac{MN}{R}
\sum_{1\le r\le R}
\sum_{\substack{
N<n\le 2N
}}
\biggl|
\sum_{\substack{
M<m\le 2M\\
m(n+r)\le x
}}
e\bigl(f_{\lambda}(m(n+r))-f_{\lambda}(mn)\bigr)
\biggr|
+
\frac{M^{2}N^{2}}{R}\\
&\ll
MN
\sup_{1\le r\le R}S_{\mathrm{II}}(r)
+
\frac{M^{2}N^{2}}{R},
\end{aligned}
\end{equation}
where the sum $S_{\mathrm{II}}(r)$ is defined by
\begin{equation}
\label{lem:TypeII:def:SII_r}
S_{\mathrm{II}}(r)
\coloneqq
\sum_{\substack{
N<n\le 2N
}}|T(r,n)|
\quad\text{with}\quad
T(r,n)
\coloneqq
\sum_{\substack{
M<m\le 2M\\
m(n+r)\le x
}}
e\bigl(f_{\lambda}(m(n+r))-f_{\lambda}(mn)\bigr).
\end{equation}

Since $f_{\lambda}$ has a period $g^{\lambda}$,
by using the exponential sum \cref{def:F_lambda}, we can write
\begin{align}
T(r,n)
&=
\sum_{0\le u<g^{\lambda}}
e(f_{\lambda}(u))
\sum_{0\le v<g^{\lambda}}
e(-f_{\lambda}(v))
\sum_{\substack{
M<m\le 2M\\
m(n+r)\le x\\
m(n+r)\equiv u\ \mod{g^{\lambda}}\\
mn\equiv v\ \mod{g^{\lambda}}
}}
1\\
&=
\sum_{0\le u<g^{\lambda}}
e(f_{\lambda}(u))
\sum_{0\le v<g^{\lambda}}
e(-f_{\lambda}(v))\\
&\hspace{0.1\textwidth}
\times
\frac{1}{g^{2\lambda}}
\sum_{0\le h<g^{\lambda}}
\sum_{0\le k<g^{\lambda}}
\sum_{\substack{
M<m\le 2M\\
m(n+r)\le x
}}
e\biggl(\frac{h(m(n+r)-u)}{g^{\lambda}}+\frac{k(mn-v)}{g^{\lambda}}\biggr)\\
&=
\sum_{0\le h<g^{\lambda}}
\sum_{0\le k<g^{\lambda}}
F_{\lambda}\biggl(\frac{h}{g^{\lambda}}\biggr)
\overline{F_{\lambda}\biggl(-\frac{k}{g^{\lambda}}\biggr)}
\sum_{\substack{
M<m\le 2M\\
m(n+r)\le x
}}
e\biggl(\frac{(h+k)n+hr}{g^{\lambda}}m\biggr).
\end{align}
Since the inner sum over $m$ is a sum of geometric sequence, we have
\begin{equation}
\label{TypeII:Tn:geometric_seq}
|T(r,n)|
\ll
\sum_{0\le h<g^{\lambda}}
\sum_{0\le k<g^{\lambda}}
\biggl|F_{\lambda}\biggl(\frac{h}{g^{\lambda}}\biggr)\biggr|
\biggl|F_{\lambda}\biggl(-\frac{k}{g^{\lambda}}\biggr)\biggr|
\min\biggl(
M,
\frac{1}{|\sin\pi\frac{(h+k)n+hr}{g^{\lambda}}|}
\biggr).
\end{equation}
By taking the sum over $n$ and recalling \cref{lem:TypeII:def:SII_r}, we have
\begin{equation}
\label{lem:TypeII:SII_r:first}
\begin{aligned}
S_{\mathrm{II}}(r)
&\ll
\sum_{0\le h<g^{\lambda}}
\sum_{0\le k<g^{\lambda}}
\biggl|F_{\lambda}\biggl(\frac{h}{g^{\lambda}}\biggr)\biggr|
\biggl|F_{\lambda}\biggl(-\frac{k}{g^{\lambda}}\biggr)\biggr|
\sum_{N<n\le2N}
\min\biggl(
M,
\frac{1}{|\sin\pi\frac{(h+k)n+hr}{g^{\lambda}}|}
\biggr).
\end{aligned}
\end{equation}
Since the term
\[
\min\biggl(
M,
\frac{1}{|\sin\pi\frac{(h+k)n+hr}{g^{\lambda}}|}
\biggr)
\]
depends only on $n\ \mod{g^{\lambda}}$, we can bound \cref{lem:TypeII:SII_r:first} as
\begin{equation}
\label{lem:TypeII:SII_r:second}
\begin{aligned}
S_{\mathrm{II}}(r)
&\ll
\biggl(\frac{N}{g^{\lambda}}+1\biggr)
\sum_{0\le h<g^{\lambda}}
\sum_{0\le k<g^{\lambda}}
\biggl|F_{\lambda}\biggl(\frac{h}{g^{\lambda}}\biggr)\biggr|
\biggl|F_{\lambda}\biggl(-\frac{k}{g^{\lambda}}\biggr)\biggr|\\
&\hspace{0.4\textwidth}
\times
\sum_{0\le n<g^{\lambda}}
\min\biggl(
M,
\frac{1}{|\sin\pi\frac{(h+k)n+hr}{g^{\lambda}}|}
\biggr)
\end{aligned}
\end{equation}
by adding excess terms if necessary. By using \cref{lem:sin_sum} in \cref{lem:TypeII:SII_r:second}, we have
\begin{equation}
\label{lem:TypeII:T:after_sin_sum}
\begin{aligned}
S_{\mathrm{II}}(r)
&\ll
\biggl(\frac{N}{g^{\lambda}}+1\biggr)
\sum_{d\mid g^{\lambda}}
d
\sum_{\substack{
0\le h,k<g^{\lambda}\\
(h+k,g^{\lambda})=d
}}
\biggl|F_{\lambda}\biggl(\frac{h}{g^{\lambda}}\biggr)\biggr|
\biggl|F_{\lambda}\biggl(-\frac{k}{g^{\lambda}}\biggr)\biggr|
\min\biggl(
M,
\frac{1}{\sin\pi\frac{d}{g^{\lambda}}\|\frac{hr}{d}\|}
\biggr)\\
&\hspace{0.05\textwidth}
+
\biggl(\frac{N}{g^{\lambda}}+1\biggr)
g^{\lambda}(\log x)
\sum_{0\le h<g^{\lambda}}
\sum_{0\le k<g^{\lambda}}
\biggl|F_{\lambda}\biggl(\frac{h}{g^{\lambda}}\biggr)\biggr|
\biggl|F_{\lambda}\biggl(-\frac{k}{g^{\lambda}}\biggr)\biggr|\\
&\eqqcolon
\biggl(\frac{N}{g^{\lambda}}+1\biggr)
(S_{\mathrm{II},1}(r)+S_{\mathrm{II},2}(r)).
\end{aligned}
\end{equation}

We first estimate the sum $S_{\mathrm{II},1}(r)$.
Since $d\mid g^{\lambda}$ in the sum $S_{\mathrm{II},1}(r)$,
we can bound $S_{\mathrm{II},1}(r)$ as
\begin{align}
S_{\mathrm{II},1}(r)
\le
\sum_{d\mid g^{\lambda}}
d
\sum_{\substack{
0\le h,k<g^{\lambda}\\
h\equiv-k\ \mod{d}
}}
\biggl|F_{\lambda}\biggl(\frac{h}{g^{\lambda}}\biggr)\biggr|
\biggl|F_{\lambda}\biggl(-\frac{k}{g^{\lambda}}\biggr)\biggr|
\min\biggl(
M,
\frac{1}{\sin\pi\frac{d}{g^{\lambda}}\|\frac{hr}{d}\|}
\biggr)
\end{align}
by weakening the condition $(h+k,g^{\lambda})=d$ to $h\equiv-k\ \mod{d}$.
Since the function
\[
\mathbb{Z}\to\mathbb{R}\semicolon
k\mapsto
\biggl|F_{\lambda}\biggl(\frac{k}{g^{\lambda}}\biggr)\biggr|
\]
has a period $g^{\lambda}$, we can change variables via $k\leadsto {-k}$ to get
\[
S_{\mathrm{II},1}(r)
\le
\sum_{d\mid g^{\lambda}}
d
\sum_{\substack{
0\le h,k<g^{\lambda}\\
h\equiv k\ \mod{d}
}}
\biggl|F_{\lambda}\biggl(\frac{h}{g^{\lambda}}\biggr)\biggr|
\biggl|F_{\lambda}\biggl(\frac{k}{g^{\lambda}}\biggr)\biggr|
\min\biggl(
M,
\frac{1}{\sin\pi\frac{d}{g^{\lambda}}\|\frac{hr}{d}\|}
\biggr).
\]
By classifying the values of $h\equiv k\ \mod{d}$, we obtain
\begin{align}
S_{\mathrm{II},1}(r)
&=
\sum_{d\mid g^{\lambda}}
d
\sum_{0\le a<d}
\min\biggl(
M,
\frac{1}{\sin\pi\frac{d}{g^{\lambda}}\|\frac{ar}{d}\|}
\biggr)
\biggl(
\sum_{\substack{
0\le h<g^{\lambda}\\
h\equiv a\ \mod{d}
}}
\biggl|F_{\lambda}\biggl(\frac{h}{g^{\lambda}}\biggr)\biggr|
\biggr)^{2}.
\end{align}
Let $v_{g}(d)$ be the largest integer satisfying $g^{v_{g}(d)}\mid d$.
By \cref{lem:F_L1_moment}, we then have
\[
S_{\mathrm{II},1}(r)
\ll
\sum_{d\mid g^{\lambda}}
d^{1-2\eta_{g}}
g^{2\eta_{g}\lambda}
\sum_{0\le a<d}
\min\biggl(
M,
\frac{1}{\sin\pi\frac{d}{g^{\lambda}}\|\frac{ar}{d}\|}
\biggr)
\biggl|F_{v_{g}(d)}^{[\lambda-v_{g}(d)]}\biggl(\frac{a}{g^{v_{g}(d)}}\biggr)\biggr|^{2}.
\]
By using \cref{lem:L_infty_general}, we have
\begin{equation}
\label{lem:TypeII:SII_1_r:before_final_sin_sum}
S_{\mathrm{II},1}(r)
\ll
\sum_{d\mid g^{\lambda}}
d^{1-2\eta_{g}}
g^{2\eta_{g}\lambda-\sigma_{v_{g}(d)}^{[\lambda-v_{g}(d)]}(\bm{\alpha})}
\sum_{0\le a<d}
\min\biggl(
M,
\frac{1}{\sin\pi\frac{d}{g^{\lambda}}\|\frac{ar}{d}\|}
\biggr).
\end{equation}
Since $x\mapsto\sin\pi x$ is concave for $x\in[0,\frac{1}{2}]$, we have
\[
\sin\biggl(\pi\frac{d}{g^{\lambda}}\biggl\|\frac{ar}{d}\biggr\|\biggr)
\ge
\frac{d}{g^{\lambda}}\sin\biggl(\pi\biggl\|\frac{ar}{d}\biggr\|\biggr)
=
\frac{d}{g^{\lambda}}\biggl|\sin\pi\frac{ar}{d}\biggr|,
\]
and so
we can apply \cref{lem:sin_sum} to the inner sum of \cref{lem:TypeII:SII_1_r:before_final_sin_sum} to obtain
\begin{align}
\sum_{0\le a<d}
\min\biggl(
M,
\frac{1}{\sin\pi\frac{d}{g^{\lambda}}\|\frac{ar}{d}\|}
\biggr)
&\le
\frac{g^{\lambda}}{d}
\sum_{0\le a<d}
\min\biggl(
\frac{dM}{g^{\lambda}},
\frac{1}{|\sin\pi\frac{ar}{d}|}
\biggr)\\
&\ll
\frac{g^{\lambda}}{d}
\biggl(
(r,d)\frac{dM}{g^{\lambda}}
+
d\log d
\biggr)
\ll
MR+g^{\lambda}\log d
\ll
g^{\lambda}\log x
\end{align}
since $g^{\lambda}\asymp MR^{2}$ by \cref{lem:TypeII:def:lambda}.
On inserting this estimate into \cref{lem:TypeII:SII_1_r:before_final_sin_sum}, we get
\[
S_{\mathrm{II},1}(r)
\ll
(\log x)
\sum_{d\mid g^{\lambda}}
d^{1-2\eta_{g}}
g^{(1+2\eta_{g})\lambda-\sigma_{v_{g}(d)}^{[\lambda-v_{g}(d)]}(\bm{\alpha})}.
\]
We now write $d=kg^{\delta}$ with $0\le\delta\le\lambda$ and $g\nmid k$ so that $v_{g}(d)=\delta$.
This gives
\begin{equation}
\label{lem:TypeII:SII_1_r:only_divisor_sum}
S_{\mathrm{II},1}(r)
\ll
(\log x)
\sum_{0\le\delta\le\lambda}
g^{(1+2\eta_{g})\lambda+\delta(1-2\eta_{g})-\sigma_{\delta}^{[\lambda-\delta]}(\bm{\alpha})}
\sum_{\substack{
k\mid g^{\lambda-\delta}\\
g\nmid k
}}
k^{1-2\eta_{g}}.
\end{equation}
For $k\mid g^{\lambda-\delta}$ with $g\nmid k$,
we have $(k,g)\le g/2$, and so
\[
k
=
(k,g^{\lambda-\delta})
\le
(k,g)^{\lambda-\delta}
\le
(g/2)^{\lambda-\delta}.
\]
This gives
\[
\sum_{\substack{
k\mid g^{\lambda-\delta}\\
g\nmid k
}}
k^{1-2\eta_{g}}
\le
(g/2)^{(1-2\eta_{g})(\lambda-\delta)}
\tau(g^{\lambda-\delta})
\ll
g^{(1-2\eta_{g})(\lambda-\delta)-\omega_{g}(\lambda-\delta)}
\]
with
\[
\omega_{g}
\coloneqq
\frac{\log 2}{\log g}\biggl(\frac{1}{2}-\eta_{g}\biggr)
\]
by using the bound $\tau(n)\ll_{\epsilon}n^{\epsilon}$.
On inserting this estimate into \cref{lem:TypeII:SII_1_r:only_divisor_sum}, we get
\begin{equation}
\label{lem:TypeII:SII_1_r:prefinal}
\begin{aligned}
S_{\mathrm{II},1}(r)
&\ll
g^{2\lambda}
(\log x)^{2}
\max_{0\le\delta\le\lambda}
g^{-\omega_{g}(\lambda-\delta)-\sigma_{\delta}^{[\lambda-\delta]}(\bm{\alpha})}
=
g^{2\lambda}
(\log x)^{2}
\max_{0\le\delta\le\lambda}
g^{-\omega_{g}\delta-\sigma_{\lambda-\delta}^{[\delta]}(\bm{\alpha})},
\end{aligned}
\end{equation}
where we changed variables via $\delta\leadsto\lambda-\delta$ in the last equality.
By \cref{lem:TypeII:def:mu} and \cref{lem:TypeII:def:lambda}, we have
\begin{equation}
\label{lem:TypeII:mu_lambda_range}
g^{\mu-1}\le M\le MR^{2}<g^{\lambda},
\quad\text{and so}\quad
\mu\in[0,\lambda]\cap\mathbb{Z}.
\end{equation}
Thus, by \cref{lem:sigma_monotonicity:with_eta} and \cref{lem:sigma_monotonicity:with_lambda}
of \cref{lem:sigma_monotonicity} and \cref{lem:TypeII:mu_range},
we have
\[
\omega_{g}\delta+\sigma_{\lambda-\delta}^{[\delta]}(\bm{\alpha})
\ge
\sigma_{\lambda}(\bm{\alpha})
\ge
\sigma_{\mu}(\bm{\alpha})
\ge
10\kappa_{\mathrm{II}}
\]
in \cref{lem:TypeII:SII_1_r:prefinal} and so
\begin{equation}
\label{lem:TypeII:SII_1_r}
S_{\mathrm{II},1}(r)
\ll
g^{2\lambda-10\kappa_{\mathrm{II}}}(\log x)^{2}.
\end{equation}

For the sum $S_{\mathrm{II},2}(r)$,
we can simply use \cref{lem:F_L1_moment:pure} of \cref{lem:F_L1_moment} to get
\begin{align}
S_{\mathrm{II},2}(r)
&=
g^{\lambda}(\log x)
\biggl(
\sum_{0\le h<g^{\lambda}}
\biggl|F_{\lambda}\biggl(\frac{h}{g^{\lambda}}\biggr)\biggr|
\biggr)^{2}
\ll
g^{(1+2\eta_{g})\lambda}\log x
=
g^{2\lambda-(1-2\eta_{g})\lambda}\log x.
\end{align}
Since
\cref{sigma_bound},
\cref{eta_bound}, 
\cref{lem:TypeII:mu_range},
\cref{lem:TypeII:mu_lambda_range} imply
\[
(1-2\eta_{g})\lambda
\ge
\biggl(\frac{1}{2}-\eta_{g}\biggr)\xi_{\mathrm{II}}
\ge
\frac{\xi_{\mathrm{II}}}{4g^{3}\log g}
\ge
10\kappa_{\mathrm{II}},
\]
we have
\begin{equation}
\label{lem:TypeII:SII_2_r}
\begin{aligned}
S_{\mathrm{II},2}(r)
\ll
g^{2\lambda-10\kappa_{\mathrm{II}}}\log x.
\end{aligned}
\end{equation}

On inserting \cref{lem:TypeII:SII_1_r} and \cref{lem:TypeII:SII_2_r}
into \cref{lem:TypeII:T:after_sin_sum}
and recalling \cref{lem:TypeII:def:lambda}, we obtain
\[
S_{\mathrm{II}}(r)
\ll
(Ng^{\lambda}+g^{2\lambda})g^{-10\kappa_{\mathrm{II}}}(\log x)^{2}
\ll
(MNR^{2}+M^{2}R^{4})g^{-10\kappa_{\mathrm{II}}}(\log x)^{2}.
\]
Since we are assuming $M\le N$, this can be bounded by
\[
S_{\mathrm{II}}(r)
\ll
MNR^{4}g^{-10\kappa_{\mathrm{II}}}(\log x)^{2}.
\]
On inserting this estimate into \cref{lem:TypeII:truncation}, we obtain
\begin{equation}
\label{lem:TypeII:prefinal}
S_{\mathrm{II}}
\ll
MNR^{2}x^{-5\kappa_{\mathrm{II}}}\log x
+
MNR^{-\frac{1}{2}}.
\end{equation}
We now choose $R=g^{2\kappa_{\mathrm{II}}}$.
By \cref{sigma_bound}, we have
\[
2\kappa_{\mathrm{II}}
\le
\frac{\xi_{\mathrm{II}}}{2g^{3}\log g}
\le
\frac{\theta}{2}\frac{\log x}{\log g},
\]
and so, by \cref{lem:TypeII:MN_range}, we have
\[
R=g^{2\kappa_{\mathrm{II}}}\le x^{\frac{\theta}{2}}\le N^{\frac{1}{2}}
\]
as required in \cref{lem:TypeII:R_range}.
On inserting $R=g^{2\kappa_{\mathrm{II}}}$ and $MN<x$ into \cref{lem:TypeII:prefinal}, we obtain the assertion.
\end{proof}

\section{Exponential sum over primes with weakly digital function}
\label{sec:exp_sum_over_prime_wdf}
We now prove an estimate for exponential sums over primes with weakly digital function.

\begin{lemma}[Vaughan's identity]
\label{lem:Vaughan}
For $z>0$ and $n\in\mathbb{N}$, we have
\[
\Lambda(n)
=
a_{1}(n)+a_{2}(n)+a_{3}(n)+a_{4}(n)
\]
with
\begin{alignat}{3}
a_{1}(n)
&\coloneqq
\sum_{\substack{
dm=n\\
d\le z
}}
\mu(d)\log m,
&\qquad
a_{2}(n)
&\coloneqq
\sum_{\substack{
dm=n\\
d,m>z
}}
\mu(d)c_{2}(m),\\
a_{3}(n)
&\coloneqq
-
\sum_{\substack{
dm=n\\
d\le z^{2}
}}
c_{3}(d),
&\qquad
a_{4}(n)
&\coloneqq
\Lambda(n)\mathbbm{1}_{n\le z},\\
c_{2}(n)
&\coloneqq
\sum_{\substack{
dm=n\\
d>z
}}
\Lambda(d),
&\qquad
c_{3}(n)
&\coloneqq
\sum_{\substack{
dm=n\\
d,m\le z
}}
\mu(d)\Lambda(m).
\end{alignat}
\end{lemma}
\begin{proof}
By using the truncated Dirichlet series
\[
M(s)
\coloneqq
\sum_{n\le z}
\frac{\mu(n)}{n^{s}}
\quad\text{and}\quad
F(s)
\coloneqq
\sum_{n\le z}
\frac{\Lambda(n)}{n^{s}},
\]
it suffices to see the identity
\[
-\frac{\zeta'}{\zeta}(s)
=
-M(s)\zeta'(s)
+\biggl(\frac{1}{\zeta(s)}-M(s)\biggr)\biggl(-\frac{\zeta'}{\zeta}(s)-F(s)\biggr)\zeta(s)
-M(s)F(s)\zeta(s)
+F(s)
\]
and compare the coefficients of the Dirichlet series on the left and right hand sides.
\end{proof}

\begin{theorem}
\label{thm:exp_sum}
For $g\in\mathbb{Z}_{\ge2}$, $\bm{\alpha}\in\mathscr{A}_{g}$, $L\in\mathbb{N}$,
$2\le x\le g^{L}$, we have
\[
S
\coloneqq
\sum_{n\le x}\Lambda(n)e(f_{L}(n))
\ll
xg^{-\kappa}(\log x)^{4},
\]
where
\begin{equation}
\kappa
\coloneqq
\frac{1}{10}
\sigma_{\xi}(\bm{\alpha})
\quad\text{with}\quad
\xi
\coloneqq
\biggl[\frac{1}{4}\frac{\log x}{\log g}\biggr]
\end{equation}
and the implicit constant depends only on $g$.
\end{theorem}
\begin{proof}
Let $z=x^{\frac{1}{4}}$.
By \cref{lem:Vaughan}, we can decompose the sum as
\begin{equation}
\label{thm:exp_sum:decomp}
S=S_{1}+S_{2}+S_{3}+S_{4},
\end{equation}
where
\begin{alignat}{3}
S_{1}
&\coloneqq
\sum_{\substack{
mn\le x\\
m\le z
}}\mu(m)(\log n)e(f_{L}(mn)),
&\qquad
S_{2}
&\coloneqq
\sum_{\substack{
mn\le x\\
m,n>z
}}
\mu(m)c_{2}(n)
e(f_{L}(mn)),\\
S_{3}
&\coloneqq
\sum_{\substack{
mn\le x\\
m\le z^{2}
}}
c_{3}(m)e(f_{L}(mn)),
&\qquad
S_{4}
&\coloneqq
\sum_{n\le z}
e(f_{L}(n)).
\end{alignat}
We estimate these sums separetely.

For the sum $S_{1}$, since $z\le x^{\frac{1}{2}}$, we can use partial summation and \cref{lem:TypeI} to get
\begin{equation}
\label{thm:exp_sum:S1}
S_{1}
\ll
(\log x)
\sum_{m\le z}
\sup_{1\le t\le x/m}
\biggl|
\sum_{n\le t}e(f_{L}(mn))
\biggr|
\ll
xg^{-\kappa}(\log x)^{3},
\end{equation}
where we used $\kappa_{\mathrm{I}}\ge\kappa$.

For the sum $S_{2}$, we can use \cref{lem:TypeII}. Note that
\begin{equation}
\label{thm:exp_sum:coefficient_bound}
|c_{2}(n)|,
|c_{3}(n)|
\le
\sum_{d\mid n}\Lambda(d)
=
\log n.
\end{equation}
We can take $\theta=\frac{1}{4}$ in \cref{lem:TypeII} so $\kappa_{\mathrm{II}}=\kappa$.
Thus, by decomposing the sum dyadically, we get
\begin{equation}
\label{thm:exp_sum:S2}
S_{2}
\ll
(\log x)^{3}
\sup_{\substack{
M,N\ge x^{\frac{1}{4}}\\
|a(m)|,|b(n)|\le1
}}
\biggl|
\sum_{\substack{
mn\le x\\
M<m\le2M\\
N<n\le2N
}}
a(m)b(n)e(f_{L}(mn))
\biggr|
\ll
xg^{-\kappa}(\log x)^{4}.
\end{equation}

For the sum $S_{3}$, we first decompose as
\begin{equation}
\label{thm:exp_sum:S3:decomp}
S_{3}
=
\sum_{\substack{
mn\le x\\
m\le z
}}
+
\sum_{\substack{
mn\le x\\
z<m\le z^{2}\\
n>z
}}
+
\sum_{\substack{
mn\le x\\
z<m\le z^{2}\\
n\le z
}}
\eqqcolon
S_{31}
+
S_{32}
+
S_{33}.
\end{equation}
By recalling \cref{thm:exp_sum:coefficient_bound},
we can apply \cref{lem:TypeI} to get
\begin{equation}
\label{thm:exp_sum:S31}
S_{31}
\ll
xg^{-\kappa}(\log x)^{3}
\end{equation}
similarly to $S_{1}$, and we can apply \cref{lem:TypeII} to get
\begin{equation}
\label{thm:exp_sum:S32}
S_{32}
\ll
xg^{-\kappa}(\log x)^{4}
\end{equation}
similarly to $S_{2}$. For the sum $S_{33}$, we just use a trivial estimate
\begin{equation}
\label{thm:exp_sum:S33}
S_{33}
\ll
z^{3}
=
x^{1-\frac{1}{4}}\log x
\ll
xg^{-\kappa}\log x,
\end{equation}
where we used the bound
\[
\kappa
\le
\frac{\xi}{20}
\le
\frac{1}{80}\frac{\log x}{\log g}
<
\frac{1}{4}\frac{\log x}{\log g}
\]
which follows by \cref{sigma_bound}.
On inserting \cref{thm:exp_sum:S31}, \cref{thm:exp_sum:S32}, \cref{thm:exp_sum:S33}
into \cref{thm:exp_sum:S3:decomp}, we obtain
\begin{equation}
\label{thm:exp_sum:S3}
S_{3}
\ll
xg^{-\kappa}(\log x)^{4}.
\end{equation}

Finally, the sum $S_{4}$ is bounded trivially as
\begin{equation}
\label{thm:exp_sum:S4}
S_{4}
\ll
z
\ll
x^{1-\frac{3}{4}}
\ll
xg^{-\kappa}
\end{equation}
similarly to $S_{33}$.

On inserting
\cref{thm:exp_sum:S1},
\cref{thm:exp_sum:S2},
\cref{thm:exp_sum:S3},
\cref{thm:exp_sum:S4}
into \cref{thm:exp_sum:decomp},  we obtain the assertion.
\end{proof}

\section{Application to digital reverse}
\label{sec:digital_reverse}
We finally apply \cref{thm:exp_sum} to digital reverse
to obtain \cref{thm:Telhcirid} and \cref{thm:Zsiflaw_Legeis}.
In order to make the digital reverse fit into the setting of \cref{thm:exp_sum},
we introduce a slightly modified digital reverse.
For $L,n\in\mathbb{Z}_{\ge0}$,
we define \textdf{the relative digital reverse} $\rev_{L}(n)$ of order $L$ by
\[
\rev_{L}(n)
\coloneqq
\sum_{0\le i<L}\epsilon_{i}(n)g^{L-i-1}.
\]
Note that then the ``absolute'' digital reverse defined by \cref{sec:intro:def:abs_digital_reverse}
can be expressed as
\[
\arev{n}
\coloneqq
\rev_{\len(n)}(n)
\]
for $n\in\mathbb{Z}_{\ge0}$.
Note that the relative and absolute digital reverses coincide
on $[g^{L-1},g^{L})\cap\mathbb{Z}$, i.e.\ 
\[
\arev{n}=\rev_{L}(n)
\quad\text{for}\quad
n\in[g^{L-1},g^{L}).
\]
To fit into the picture of weakly digital function,
consider $\bm{\alpha}_{L}=(\alpha_{L,i})_{i=0}^{\infty}\in\mathscr{A}_{g}$ given by
\[
\alpha_{L,i}(n)
\coloneqq
\alpha ng^{L-i-1}
\]
for $\alpha\in\mathbb{R}$. This seed $\bm{\alpha}_{L}$ gives
\[
\rev_{L}(n)
=
\alpha f_{L,\bm{\alpha}_{L}}(n).
\]
Note that $\bm{\alpha}_{L}$ depends on a parameter $L$,
the uniformity over $\bm{\alpha}\in\mathscr{A}_{g}$ in \cref{thm:exp_sum} is necessary.

\begin{lemma}
\label{lem:geometric_progression_mod1}
For $g\in\mathbb{Z}_{\ge2}$ and $\alpha\in\mathbb{R}\setminus\mathbb{Z}$, we have
\[
\|g^{i_{0}}\alpha\|\ge\frac{1}{g+1}
\quad\text{with}\quad
i_{0}
\coloneqq
\biggl[\frac{\log\frac{g}{(g+1)\|\alpha\|}}{\log g}\biggr].
\]
\end{lemma}
\begin{proof}
This is Lemma~2.7 of \cite{DMRSS:ReversiblePrime}.
For completeness, we give a proof. We have
\[
g^{i_{0}}\|\alpha\|
\ge
g^{\frac{\log\frac{g}{(g+1)\|\alpha\|}}{\log g}-1}\|\alpha\|
=
\frac{1}{g}\frac{g}{(g+1)\|\alpha\|}\cdot\|\alpha\|
=
\frac{1}{g+1}
\]
and
\[
g^{i_{0}}\|\alpha\|
\le
g^{\frac{\log\frac{g}{(g+1)\|\alpha\|}}{\log g}}\|\alpha\|
\le
\frac{g}{(g+1)\|\alpha\|}\cdot\|\alpha\|
=
1-\frac{1}{g+1}.
\]
The result now follows since $\|g^{i_{0}}\alpha\|=\|g^{i_{0}}\|\alpha\|\|$.
\end{proof}

\begin{lemma}
\label{lem:reverse:sigma_lower}
Consider $g\in\mathbb{Z}_{\ge2}$, $L\in\mathbb{Z}_{\ge0}$ and $\alpha\in\mathbb{R}$ satisfying
\begin{equation}
\label{lem:reverse:sigma_lower:alpha_cond}
\sigma
\coloneqq
\min_{0\le i\le L}\|g^{i}(g^{2}-1)\alpha\|
>0.
\end{equation}
Define $\bm{\alpha}_{L}=(\alpha_{L,i})\in\mathscr{A}_{g}$ by
$\alpha_{L,i}(n)\coloneqq\alpha ng^{L-i-1}$.
For $\lambda\in\mathbb{Z}_{\ge0}$ with $0\le\lambda\le L$,
we have
\[
\sigma_{\lambda}(\bm{\alpha}_{L})
\gg
\frac{\lambda}{\log\frac{1}{\sigma}}+O(1),
\]
where the implicit constants depend only on $g$.
\end{lemma}
\begin{proof}
By considering the contribution of $(m,n)=(0,1)$ to \cref{def:gamma},
for $0\le i<\lambda$, we have
\begin{align}
\gamma_{i}(\bm{\alpha}_{L})
&\gg
\|(g\alpha_{L,i}(0)-\alpha_{L,i+1}(0))
-
(g\alpha_{L,i}(1)-\alpha_{L,i+1}(1))\|^{2}\\
&=
\|(g^{L-i}-g^{L-(i+1)-1})\alpha\|^{2}
=
\|g^{L-i-2}(g^{2}-1)\alpha\|^{2}
\gg
\|g^{L-i-1}(g^{2}-1)\alpha\|^{2},
\end{align}
where we used the inequality $g\|x/g\|\ge\|x\|$ at the last inequality.
We thus have
\begin{equation}
\label{lem:reverse:sigma_lower:sigma_to_sum}
\begin{aligned}
\sigma_{\lambda}(\bm{\alpha}_{L})
=
\sum_{0\le i<\lambda}\gamma_{i}(\bm{\alpha}_{L})
&\gg
\sum_{0\le i<\lambda}
\|g^{L-i-1}(g^{2}-1)\alpha\|^{2}
=
\sum_{L-\lambda\le i<L}
\|g^{i}(g^{2}-1)\alpha\|^{2}.
\end{aligned}
\end{equation}
Let
\[
J
\coloneqq
1+\biggl[\frac{\log\frac{g}{(g+1)\sigma}}{\log g}\biggr]\ge1,\quad
K
\coloneqq
\biggl[\frac{\lambda}{J}\biggr]
\and
\alpha_{k}
\coloneqq
g^{L-\lambda+(k-1)J}(g^{2}-1)\alpha.
\]
We then have
\begin{equation}
\label{lem:reverse:sigma_lower:inner_sum}
\begin{aligned}
\sum_{L-\lambda\le i<L}
\|(g^{2}-1)g^{i}\alpha\|^{2}
&\ge
\sum_{1\le k\le K}
\sum_{L-\lambda+(k-1)J\le i<L-\lambda+kJ}
\|g^{i}(g^{2}-1)\alpha\|^{2}\\
&\ge
\sum_{1\le k\le K}
\sum_{0\le i<J}
\|g^{i}\alpha_{k}\|^{2}.
\end{aligned}
\end{equation}
By the assumption \cref{lem:reverse:sigma_lower:alpha_cond} and $0\le\lambda\le L$,
we have $\alpha_{k}\not\in\mathbb{Z}$ for all $k\in\{1,\ldots,K\}$
and so \cref{lem:geometric_progression_mod1}
is applicable to the inner sum with $\alpha\coloneqq\alpha_{k}$.
Also, \cref{lem:reverse:sigma_lower:alpha_cond} implies $\|\alpha_{k}\|\ge\sigma$.
For $0\le k<K$, let
\[
i_{0,k}
\coloneqq
\biggl[\frac{\log\frac{g}{(g+1)\|\alpha_{k}\|}}{\log g}\biggr].
\]
We then have
\[
0\le i_{0,k}\le\biggl[\frac{\log\frac{g}{(g+1)\sigma}}{\log g}\biggr]<J.
\]
Therefore, we can pick up the contribution of $i=i_{0,k}$
and use \cref{lem:geometric_progression_mod1} in \cref{lem:reverse:sigma_lower:inner_sum} to get
\begin{equation}
\label{lem:reverse:sigma_lower:inner_sum_lower_bound}
\sum_{L-\lambda\le i<L}
\|(g^{2}-1)g^{i}\alpha\|^{2}
\ge
\sum_{1\le k\le K}
\|g^{i_{0,k}}\alpha_{k}\|^{2}
\ge
\frac{K}{(g+1)^{2}}
\gg
K
\end{equation}
By using the estimate
\[
K
=
\frac{\lambda}{J}+O(1)
\ge
\frac{\lambda}{1+\frac{\log\frac{1}{\sigma}}{\log 2}}
+O(1)
\gg
\frac{\lambda}{\log\frac{1}{\sigma}}
+O(1)
\]
in \cref{lem:reverse:sigma_lower:inner_sum_lower_bound}, we get
\[
\sum_{L-\lambda\le i<L}
\|(g^{2}-1)g^{i}\alpha\|^{2}
\gg
\frac{\lambda}{\log\frac{1}{\sigma}}
+
O(1).
\]
On inserting this estimate into \cref{lem:reverse:sigma_lower:sigma_to_sum},
we obtain the result.
\end{proof}

For $g\in\mathbb{Z}_{\ge2}$, $L\in\mathbb{Z}_{\ge0}$, $a,q\in\mathbb{Z}$
with $q\ge1$ and $x\ge1$, let us write
\begin{alignat}{2}
\arev{\psi}\!_{L}(x,a,q)
&\coloneqq
\sum_{\substack{
n\le x\\
\rev_{L}(n)\equiv a\ \mod{q}
}}
\Lambda(n),\qquad
&\arev{\psi}\!_{\sharp,L}(x,a,q)
&\coloneqq
\sum_{\substack{
n\le x\\
\rev_{L}(n)\equiv a\ \mod{(q,g^{L}(g^{2}-1))}
}}
\Lambda(n),\\
\arev{\vartheta}\!_{L}(x,a,q)
&\coloneqq
\sum_{\substack{
p\le x\\
\rev_{L}(p)\equiv a\ \mod{q}
}}
\log p,\qquad
&\arev{\vartheta}\!_{\sharp,L}(x,a,q)
&\coloneqq
\sum_{\substack{
p\le x\\
\rev_{L}(p)\equiv a\ \mod{(q,g^{L}(g^{2}-1))}
}}
\log p,\\
\arev{\pi}\!_{L}(x,a,q)
&\coloneqq
\sum_{\substack{
p\le x\\
\rev_{L}(p)\equiv a\ \mod{q}
}}
1,\qquad
&\arev{\pi}\!_{\sharp,L}(x,a,q)
&\coloneqq
\sum_{\substack{
p\le x\\
\rev_{L}(p)\equiv a\ \mod{(q,g^{L}(g^{2}-1))}
}}
1.
\end{alignat}

\begin{theorem}
\label{thm:rev_psi_theta_pi_L}
For $g\in\mathbb{Z}_{\ge2}$, $L,a,q\in\mathbb{Z}$, $x\in\mathbb{R}$
with $L,q\ge1$ and $1\le x\le g^{L}$,
we have
\begin{align}
\label{thm:rev_psi_theta_pi_L:psi}
\arev{\psi}\!_{L}(x,a,q)
&=
\frac{(q,g^{L}(g^{2}-1))}{q}
\arev{\psi}\!_{\sharp,L}(x,a,q)
+
O\biggl(x\exp\biggl(-c\cdot\frac{\log x}{\log(q+1)}\biggr)\biggr),\\[2mm]
\label{thm:rev_psi_theta_pi_L:theta}
\arev{\vartheta}\!_{L}(x,a,q)
&=
\frac{(q,g^{L}(g^{2}-1))}{q}
\arev{\vartheta}\!_{\sharp,L}(x,a,q)
+
O\biggl(x\exp\biggl(-c\cdot\frac{\log x}{\log(q+1)}\biggr)\biggr),\\[2mm]
\label{thm:rev_psi_theta_pi_L:pi}
\arev{\pi}\!_{L}(x,a,q)
&=
\frac{(q,g^{L}(g^{2}-1))}{q}
\arev{\pi}\!_{\sharp,L}(x,a,q)
+
O\biggl(x\exp\biggl(-c\cdot\frac{\log x}{\log(q+1)}\biggr)\biggr)
\end{align}
provided
\begin{equation}
\label{thm:rev_psi_theta_pi_L:q_range}
q\le\exp\biggl(c\cdot\frac{\log x}{\log\log x}\biggr)
\end{equation}
with some constant $c>0$,
where $c$ and the implicit constant depend only on $g$.
\end{theorem}
\begin{proof}
We may assume $x$ is large and $q\ge2$ since otherwise the assertions are trivial.

We first prove \cref{thm:rev_psi_theta_pi_L:psi}.
By the orthogonality of additive characters, we have
\begin{equation}
\label{thm:rev_psi_theta_pi_L:decomp}
\begin{aligned}
\arev{\psi}\!_{L}(x,a,q)
&=
\frac{1}{q}
\sum_{0\le h<q}
e\biggl(-\frac{ha}{q}\biggr)
\sum_{n\le x}
\Lambda(n)
e\biggl(\frac{h\rev_{L}(n)}{q}\biggr)\\
&=
\sum_{q\mid g^{L}(g^{2}-1)h}
+
\sum_{q\nmid g^{L}(g^{2}-1)h}
\eqqcolon
S_{\mathfrak{M}}+S_{\mathfrak{m}}.
\end{aligned}
\end{equation}
For the sum $S_{\mathfrak{M}}$, by using
\[
q\mid g^{L}(g^{2}-1)h
\iff
\frac{q}{(q,g^{L}(g^{2}-1))}\mid h,
\]
by rewriting $h$ to $\frac{q}{(q,g^{L}(g^{2}-1))}h$ and using the orthogonality, we have
\begin{equation}
\label{thm:rev_psi_theta_pi_L:S_Major}
\begin{aligned}
S_{\mathfrak{M}}
&=
\frac{1}{q}
\sum_{0\le h<(q,g^{L}(g^{2}-1))}
e\biggl(-\frac{ha}{(q,g^{L}(g^{2}-1))}\biggr)
\sum_{n\le x}
\Lambda(n)
e\biggl(\frac{h\rev_{L}(n)}{(q,g^{L}(g^{2}-1))}\biggr)\\
&=
\frac{(q,g^{L}(g^{2}-1))}{q}
\arev{\psi}\!_{\sharp,L}(x,a,q).
\end{aligned}
\end{equation}
For $S_{\mathfrak{m}}$, we just bound as
\begin{equation}
\label{thm:rev_psi_theta_pi_L:S_minor:abs}
S_{\mathfrak{m}}
\le
\max_{\substack{
0\le h<q\\
q\nmid g^{L}(g^{2}-1)h
}}
\biggl|
\sum_{n\le x}
\Lambda(n)
e\biggl(\frac{h\rev_{L}(n)}{q}\biggr)
\biggr|.
\end{equation}
We then use \cref{thm:exp_sum} with $\bm{\alpha}_{L}=(\alpha_{L,i})\in\mathscr{A}_{g}$ given by
\[
\alpha_{L,i}(n)
\coloneqq
\alpha ng^{L-i-1}
\quad\text{with}\quad
\alpha
\coloneqq
\frac{h}{q}
\]
This $\bm{\alpha}_{L}$ gives
\[
f_{L,\bm{\alpha}_{L}}(n)
=
\frac{h}{q}\sum_{0\le i<L}\epsilon_{i}(n)g^{L-i-1}
=
\frac{h\rev_{L}(n)}{q}.
\]
Also, since $q\nmid g^{L}(g^{2}-1)h$ in $S_{\mathfrak{m}}$, we have
\[
\min_{0\le i\le L}\biggl\|g^{i}(g^{2}-1)\cdot\frac{h}{q}\biggr\|
\ge
\frac{1}{q}
\]
for the quantity $\sigma$ defined by \cref{lem:reverse:sigma_lower:alpha_cond}.
Thus, by \cref{lem:reverse:sigma_lower},
$\kappa$ in \cref{thm:exp_sum} satisfies
\begin{align}
\kappa
=
\frac{1}{10}\sigma_{\xi}(\bm{\alpha})
\gg
\frac{\xi}{\log\frac{1}{\sigma}}+O(1)
\gg
\frac{\log x}{\log q}+O(1),
\quad\text{where}\quad
\xi\coloneqq\biggl[\frac{1}{4}\frac{\log x}{\log g}\biggr].
\end{align}
Thus, by applying \cref{thm:exp_sum} to \cref{thm:rev_psi_theta_pi_L:S_minor:abs}, we obtain
\begin{equation}
\label{thm:rev_psi_theta_pi_L:S_minor:prefinal}
S_{\mathfrak{m}}
\ll
x(\log x)^{4}\exp\biggl(-5c\cdot\frac{\log x}{\log q}\biggr).
\end{equation}
Under the condition \cref{thm:rev_psi_theta_pi_L:q_range}, we have
\begin{equation}
\label{thm:rev_psi_theta_pi_L:log_bound}
(\log x)^{4}
=
\exp(4\log\log x)
\le
\exp\biggl(4c\cdot\frac{\log x}{\log q}\biggr).
\end{equation}
On inserting this estimate into \cref{thm:rev_psi_theta_pi_L:S_minor:prefinal}, we obtain
\begin{equation}
\label{thm:rev_psi_theta_pi_L:S_minor}
S_{\mathfrak{m}}
\ll
x\exp\biggl(-c\cdot\frac{\log x}{\log q}\biggr).
\end{equation}
On inserting \cref{thm:rev_psi_theta_pi_L:S_Major}
and \cref{thm:rev_psi_theta_pi_L:S_minor}
into \cref{thm:rev_psi_theta_pi_L:decomp}, we obtain \cref{thm:rev_psi_theta_pi_L:psi}.

We next deduce \cref{thm:rev_psi_theta_pi_L:theta} from \cref{thm:rev_psi_theta_pi_L:psi}.
In \cref{thm:rev_psi_theta_pi_L:psi}, we may assume $c\in(0,\frac{\log 2}{2})$ to assure
\[
x\exp\biggl(-c\cdot\frac{\log x}{\log(q+1)}\biggr)
\ge
x^{\frac{1}{2}}.
\]
We can then use
\[
\sum_{\substack{
p^{v}\le x\\
v\ge 2
}}
\log p
\ll
x^{\frac{1}{2}}
\]
in \cref{thm:rev_psi_theta_pi_L:psi} to get \cref{thm:rev_psi_theta_pi_L:theta}.

We finally deduce \cref{thm:rev_psi_theta_pi_L:pi} from \cref{thm:rev_psi_theta_pi_L:theta}.
We first change the value of $c$ to replace $c$ in \cref{thm:rev_psi_theta_pi_L:theta}
and the associated \cref{thm:rev_psi_theta_pi_L:q_range} by $4c$.
(We however keep the same $c$ for \cref{thm:rev_psi_theta_pi_L:q_range}
associated to \cref{thm:rev_psi_theta_pi_L:pi}.)
We may also assume that $c\in(0,\frac{\log 2}{8})$.
Then, for $x^{\frac{1}{2}}\le u\le x$, we have $1\le x\le g^{L}$ and
\[
q
\le
\exp\biggl(c\cdot\frac{\log x}{\log\log x}\biggr)
\le
\exp\biggl(4c\cdot\frac{\log u}{\log\log u}\biggr).
\]
Thus, we can apply \cref{thm:rev_psi_theta_pi_L:theta} with replacing $x$ by $u$ to get
\[
\arev{\vartheta}\!_{L}(u,a,q)
=
\frac{(q,g^{L}(g^{2}-1))}{q}
\arev{\vartheta}\!_{\sharp,L}(u,a,q)
+
O\biggl(x\exp\biggl(-c\cdot\frac{\log x}{\log(q+1)}\biggr)\biggr)
\]
for $1\le u\le x$ since this is trivial if $u\le x^{\frac{1}{2}}$.
We can then use partial summation to deduce \cref{thm:rev_psi_theta_pi_L:pi}.
\end{proof}

Let
\begin{equation}
\overleftarrow{\pi}(x,a,q)
\coloneqq
\sum_{\substack{
p\le x\\
\rev(p)\equiv a\ \mod{q}
}}
1
\quad\text{and}\quad
\overleftarrow{\pi}\!_{\sharp}(x,a,q)
\coloneqq
\sum_{\substack{
p\le x\\
\rev(p)\equiv a\ \mod{(q,g^{L}(g^{2}-1))}
}}
1
\end{equation}

\begin{theorem}
\label{thm:Zsiflaw_Legeis_pure}
For $g,a,q\in\mathbb{Z}$ and $x\ge1$ with $g\ge2$ and $q\ge1$, there is a constant $c\in(0,1)$ such that
\[
\overleftarrow{\pi}(x,a,q)
=
\frac{(q,(g^{2}-1)g^{L})}{q}
\overleftarrow{\pi}\!_{\sharp}(x,a,q)
+
O\biggl(x\exp\biggl(-c\cdot\frac{\log x}{\log(q+1)}\biggr)\biggr)
\]
provided
\begin{equation}
\label{thm:Zsiflaw_Legeis_pure:q_range}
q\le\exp\biggl(c\cdot\frac{\log x}{\log\log x}\biggr),
\end{equation}
where $L\coloneqq[\frac{\log x}{\log g}]+1$
and $c$ and the implicit constant depend only on $g$.
\end{theorem}
\begin{proof}
We may assume $x$ is large in terms of $g$ since otherwise the assertion is trivial.
Let $\widetilde{c}$ be the constant $c$ of \cref{thm:rev_psi_theta_pi_L}
and let $c\coloneqq\frac{\widetilde{c}}{5}$.
Since
\begin{equation}
\label{thm:Zsiflaw_Legeis_pure:absolute_to_relative_rule}
g^{\ell-1}\le n<g^{\ell}
\implies
\rev_{\ell}(n)=\rev(n)
\end{equation}
for any $\ell\in\mathbb{N}$, we have
\begin{equation}
\label{thm:Zsiflaw_Legeis_pure:absolute_to_relative}
\overleftarrow{\pi}(u,a,q)-\overleftarrow{\pi}(g^{\ell-1}-1,a,q)
=
\overleftarrow{\pi}\!_{\ell}(u,a,q)-\overleftarrow{\pi}\!_{\ell}(g^{\ell-1}-1,a,q)
\end{equation}
for $g^{\ell-1}\le u<g^{\ell}$.
For $x^{\frac{1}{2}}\le g^{\ell-1}\le u<g^{\ell}$, we have
\[
q
\le
\exp\biggl(c\cdot\frac{\log x}{\log\log x}\biggr)
\le
\exp\biggl(\widetilde{c}\cdot\frac{\log u}{\log\log u}\biggr).
\]
Thus, we can apply \cref{thm:rev_psi_theta_pi_L:pi} of \cref{thm:rev_psi_theta_pi_L}
to \cref{thm:Zsiflaw_Legeis_pure:absolute_to_relative}
and use \cref{thm:Zsiflaw_Legeis_pure:absolute_to_relative_rule} again to get
\begin{equation}
\label{thm:Zsiflaw_Legeis_pure:ell_partial_sum}
\begin{aligned}
&\overleftarrow{\pi}(u,a,q)-\overleftarrow{\pi}(g^{\ell-1}-1,a,q)\\
&=
\frac{(q,g^{\ell}(g^{2}-1))}{q}
\sum_{\substack{
g^{\ell-1}\le p\le u\\
\rev(p)\equiv a\ \mod{(q,g^{\ell}(g^{2}-1))}
}}
1
+
O\biggl(x\exp\biggl(-\widetilde{c}\cdot\frac{\log x}{\log q}\biggr)\biggr)
\end{aligned}
\end{equation}
for $g^{\ell-1}\le u<g^{\ell}$ with $x^{\frac{1}{2}}\le g^{\ell-1}$ and $g^{\ell}\le x$.
By \cref{thm:Zsiflaw_Legeis_pure:q_range} and $\ell\gg\log x$, we indeed have
\[
(q,g^{\ell}(g^{2}-1))
=
(q,g^{L}(g^{2}-1))
\]
for all $\ell$ with $x^{\frac{1}{2}}\le g^{\ell-1}$ and $g^{\ell}\le x$.
Thus, we can rewrite \cref{thm:Zsiflaw_Legeis_pure:ell_partial_sum}
with replacing $\ell$ by $L$ to get
\begin{equation}
\label{thm:Zsiflaw_Legeis_pure:L_partial_sum}
\begin{aligned}
&\overleftarrow{\pi}(u,a,q)-\overleftarrow{\pi}(g^{\ell-1}-1,a,q)\\
&=
\frac{(q,g^{L}(g^{2}-1))}{q}
\sum_{\substack{
g^{\ell-1}\le p\le u\\
\rev(p)\equiv a\ \mod{(q,g^{L}(g^{2}-1))}
}}
1
+
O\biggl(x\exp\biggl(-\widetilde{c}\cdot\frac{\log x}{\log q}\biggr)\biggr)
\end{aligned}
\end{equation}
for $g^{\ell-1}\le u<\min(g^{\ell},x)$ for any $\ell\in\mathbb{N}$
since this is trivial if $u\le x^{\frac{1}{2}}$.
Thus, by summing up \cref{thm:Zsiflaw_Legeis_pure:L_partial_sum}
over $\ell$ and using the bound similar to \cref{thm:rev_psi_theta_pi_L:log_bound},
we obtain the theorem.
\end{proof}

Finally, \cref{thm:Zsiflaw_Legeis} and \cref{thm:Telhcirid} follows from \cref{thm:Zsiflaw_Legeis_pure}
by the argument of Section~11 of \cite{BhowmikSuzuki:Telhcirid}.

\subsection*{Acknowledgements}
We would like to thank Shashi Chourasiya and Daniel R. Johnston
for communicating  their result to us and for announcing our work in their paper.
The second author would like to express his gratitude to Laboratoire
Paul Painlev\'{e}, Daniel Duverney and Emmanuel Fricain for their generous hospitality in Lille in March 2025.
The second author was supported
by JSPS KAKENHI Grant Number JP21K13772
and  both authors were supported by Labex C$^{2}$EMPI (ANR-11-LABX-0007-01).



\vspace{10mm}

\noindent
{\textsc{%
\small
Gautami Bhowmik\\[.3em]
\footnotesize
Laboratoire Paul Painlev\'{e}, Labex-CEMPI, Universit\'{e} de Lille\\
59655 Villeneuve d'Ascq Cedex, France.
}

\noindent
\small
\textit{Email address}: \texttt{gautami.bhowmik@univ-lille.fr}
}
\bigskip

\noindent
{\textsc{%
\small
Yuta Suzuki\\[.3em]
\footnotesize
Department of Mathematics, Rikkyo University,\\
3-34-1 Nishi-Ikebukuro, Toshima-ku, Tokyo 171-8501, Japan.
}

\noindent
\small
\textit{Email address}: \texttt{suzuyu@rikkyo.ac.jp}
}

\end{document}